\documentclass[11pt,a4paper]{article}
\usepackage[OT1]{fontenc}
\usepackage[all]{xy}

\usepackage[english]{babel}
\usepackage{amsmath}
\usepackage{amsfonts}
\usepackage{amsthm}
\def\hh{\mathbb{V}}

\def\spa{{\cal U}(\theta_H)}
\def\spb{{\cal U}(\theta_D)}
\def\uspa{{\cal U}_U(\theta_H)}
\def\uspb{{\cal U}_U(\theta_D)}
\def\d{ {\cal D} }

\def\a{ {\cal A} }
\def\h{ {\cal H} }
\def\ele{ {\cal L} }
\def\b{ {\cal B} }
\def\u{ {\cal U} }
\def\ii{ {\cal I} }
\def\t{ {\cal T} }
\def\s{ {\cal S} }
\def\e{ {\cal E} }

\def\k{ {\cal K} }

\def\x{ {\cal X} }
\def\v{ {\cal V} }
\def\w{ {\cal W} }
\def\oo{ {\cal O} }

\def\q{ {\cal Q} }

\def\g1{ \mathfrak{g}_1  }

\def\vp{  \varphi }

\newtheorem{teo}{Theorem}[section]
\newtheorem{prop}[teo]{Proposition}
\newtheorem{lem}[teo]{Lemma}
\newtheorem{coro}[teo]{Corollary}
\newtheorem{defi}[teo]{Definition}
\theoremstyle{definition}
\newtheorem{rem}[teo]{Remark}
\newtheorem{ejem}[teo]{Example}

\title{Poincar\'e half-space of a $C^*$-algebra}
\author{E. Andruchow, G. Corach, L. Recht}

\begin{document}

\maketitle 

\begin{abstract}
Let  $\a$ be a C$^*$-algebra. Given a representation $\a\subset\b(\ele)$ in a Hilbert space $\ele$, the set $G^+\subset\a$ of positive invertible elements can be thought as the set of inner products in $\ele$, related to $\a$, which are equivalent to the original inner product. The set $G^+$ has a rich  geometry, it is a homogeneous space of the invertible group $G$ of $\a$, with an invariant Finsler metric. In the present paper we study the tangent bundle $TG^+$ of $G^+$, as a homogenous Finsler space of a natural  group of invertible matrices in $M_2(\a)$, identifying $TG^+$ with the {\it Poincar\'e halfspace} $\h$ of $\a$,
$$
\h=\{h\in\a: Im(h)\ge 0, Im(h) \hbox{ invertible}\}.
$$
We show that  $\h\simeq TG^+$ has properties similar to  those of a space of non-positive constant curvature.

\end{abstract}
\bigskip



\section{Introduction}
Let $\a$ be a unital C$^*$-algebra, $G$ the group of invertible elements in $\a$, $G^+$ the subset of $G$ of positive elements.

Observe that $G^+$, as an open subset of $\a_s:=\{x \in \a: x^*=x\}$, is an open submanifold of $\a_s$ and its tangent space at any point is identified with $\a_s$. If $G^s=G\cap \a_s$, then it can be proven that $G^+$ is the component of the identity in $G^s$. Also, there is a left action of $G$ on $G^s$ given by $g\cdot a=(g^{-1})^*ag^{-1}$, and $G^+$ is the orbit of $1$ under this action. With this dual nature, $G^+$ carries a natural structure of homogeneous space of $G$, with a linear connection, and   a Finsler metric. These facts, and many others concerning the differential geometry of $G^+$, have been studied in  \cite{cprIEOT}, \cite{cprIJM}, \cite{cprILLINOIS}. The goal of the present paper is the study of the tangent bundle $TG^+$, in particular its presentation as a homogeneous space of the unitary group of a natural quadratic form in $\a^2=\a \times \a$.

There are several motivations for this study.

We consider $G^+$ as the configuration space of a quantum mechanical system whose elements represent a family of equivalent metrics over a Hilbert space. The tangent bundle $TG^+$ is the phase space of the configuration space $G^+$ of the quantum system.  Note also that $TG^+$ identifies in a natural way with the bundle of observables associated with the different metrics. We shall explain this with more detail below. 

Note the correspondence 
$$
(a,X) \longleftrightarrow X+ i a,
$$
where  $a\in G^+$  and  $X\in (TG^+)_a$;  here $X$ is a selfadjoint element of $\a$, because $G^+$ is open in the space of selfadjoint elements of $\a$. This correspondence establishes a clear identification  between  $TG^+$ and $\h$, the {\it Poincar\'e half-space} of $\a$,
$$
\h=\{h\in\a: Im(h)\in G^+\}.
$$
Following ideas from C.L. Siegel \cite{siegel1}, \cite{siegel2} and \cite{siegel3}, we claim that there is a natural form $\theta_H$ in $\a\times\a=\a^2$, determined by $\h$, whose unitary group $\u(\theta_H)$  acts transitively in $\h$. Namely, if we put projective coordinates $\left(\begin{array}{l} x_1 \\ x_2 \end{array}\right)\in\a^2$, elements $h=x_2x_1^{-1}\in\h$ are characterized by the condition
$$
Im(h)=\frac{1}{2i}\{x_2x_1^{-1}-(x_1^*)^{-1}x_2^*\}=\frac{1}{2i}(x_1^*)^{-1}\{x_1^*x_2-x_2^*x_1\}x_1^{-1}\in G^+,
$$
or equivalently $\frac{1}{2i}\{x_1^*x_2-x_2^*x_1\}\in G^+$.
Thus, if we put
$$
\theta_H(\left(\begin{array}{c} x_1\\ x_2 \end{array} \right),\left(\begin{array}{c} y_1\\ y_2 \end{array} \right))=\frac{1}{i}\{x_1^*y_2-x_2^*y_1\},
$$
the condition $Im(h)\in G^+$ is $\theta_H(\left(\begin{array}{c} x_1\\ x_2 \end{array} \right),\left(\begin{array}{c} x_1\\ x_2 \end{array} \right))\in G^+$.

The unitary  group $\spa$ of $\theta_H$, i.e., the group of invertible matrices in $M_2(\a)$ which preserve $\theta_H$,  acts transitively on $\h$, and  makes $\h$  a homogeneous space. Therefore, $TG^+$ can be considered as the phase space of the mentioned quantum system, with the group $\u(\theta_H)$ of symmetries.

The homogenous space $\h$ identifies in turn, also in a natural fashion, with the space $\q_{\rho_H}$ of selfadjoint projections in $\a^2$ which decompose the quadratic form $\theta_H$ (where $\rho_H$ is the reflection  in $\a^2$ induced  by $\theta_H$). This identification is equivariant with the respective actions of the symmetry group $\u(\theta_H)$.

The space $\q_\rho$, for arbitrary symmetries $\rho$,  was studied thoroughly in \cite{cpr}; we can deduce the geometric and metric properties of $TG^+$ from the properties established for $\q_{\rho_H}$. 

In a forthcoming second part of this paper, we shall study  additional geometric structures related to a pre-quantization of $TG^+$, in the framework of a Hilbert-C$^*$-module structure.  

We  also establish a natural bijection between  $TG^+$ and the space $\d$ of strict contractions of $\a$,
$$
\d=\{a\in\a: \|a\| < 1\}.
$$

Besides the form $\theta_H$ and  the group $\uspa$ of invertible elements in $M_2(\a)$ which leave $\theta_H$ invariant, an important role will be played by the unit sphere $\k_H$,
$$
\k_H=\{\left(\begin{array}{c} x_1\\ x_2 \end{array} \right)\in\a^2: \theta_H(\left(\begin{array}{c} x_1\\ x_2 \end{array} \right),\left(\begin{array}{c} x_1\\ x_2 \end{array} \right))=1\}.
$$
The sphere $\k_H$ will serve the role of a coordinate space for $\h$, and these data will be related by the commutative diagram of $\u(\theta_H)$-homogeneous spaces
$$
\xymatrix{
& \k_H \ar[ld]_{\varphi}\ar[rd]^{\bar{\varphi}} \\
\h \ar[rr]^{\Phi_H}
&& \q_{\rho_H}
}$$
where $\varphi$ and $\tilde{\varphi}$ are submersions and $\Phi_H$ is a diffeomorphism.

There is an analogous commutative triangle for the open unit disk  model, by means of the form $\theta_D$,
$$
\theta_D(\left(\begin{array}{c} x_1\\ x_2 \end{array} \right),\left(\begin{array}{c} y_1\\ y_2 \end{array} \right))=x_1^*y_1-x_2^*y_2.
$$
Both forms (and therefore both sets of data) are covariantly related by the unitary matrix in $M_2(\a)$
$$
U=\frac{1}{\sqrt{2}}\left( \begin{array}{cc} 1 & 1 \\ i & -i \end{array} \right).
$$
For instance, the bijection between $\h$ and $\d$ is a Moebius transformation. 

Let us summarize the contents of the paper. In Section 2 we define the space $\h$, the form $\theta_H$, the unitary group $\spa$  and the sphere $\k_H$ of this form. In Section 3 we introduce an alternative model for $\h$, namely the open unit disk $\d$, along with the form $\theta_D$ and the unitary group $\spb$ and sphere $\k_D$ of this form. The spaces, forms, groups and unit spheres of both models are intertwined by the unitary matrix $U\in M_2(\a)$. The reason to have two models for the same space, is that some computations are easier or more natural with one or the other. For instance, in Section 4 we study the local (Banach-Lie group) structure of $\spa$. An important role is played by the subgroup $\b\subset\spa$, which we call the {\it Borel} subgroup of $\spa$, and which has a natural meaning in this setting.  In  Section 5, we introduce the action of $\spb$ on the unit sphere $\k_D$, which implies that also $\spa$ acts on $\k_H$ (this is another example, where a fact is easier to establish in model than the other). These actions enable us to introduce the actions of $\spa$ and $\spb$ on $\h$ and $\d$, respectively. It is shown that these actions are transitive, and the isotropy subgroups are computed. In Section 7, we recall from \cite{cpr}  the space  $\q_\rho$ of signed decompositions of a form induced by a symmetry $\rho$, and prove that there are natural diffeomorphisms between $\q_\rho$, $\h$ and $\d$, which are equivariant with respect to the corresponding group actions. This is a key fact, which allows us to import from $\q_\rho$ to $\h$ and $\d$ the main geometric features of that space: a linear connection, a Finsler metric and its properties. Among these, that $\h$ behaves as a non-positively curved metric length space \cite{gromov}. In Section 9 we compute the specific form of the linear connection induced in $\h$. In Section 10 we compute special cases of geodesics in $\h$. Finally, in Section 11, as an Appendix, we outline an intrinsic, coordinate free, version for the Poncar\'e half-space, in terms of a Hilbertizable space endowed with a coherent family of inner products.

\section{Poincar\'e halfspace}

We define the following forms in $\a^2$:
\begin{defi}
The $\a$-valued {\it inner product}:
$$
<\left(\begin{array}{c} a_1\\ a_2 \end{array} \right),\left(\begin{array}{c} b_1\\ b_2 \end{array} \right)>=a_1^*b_1+a_2^*b_2
$$
and the $\a$-valued {\it simplectic form}
$$
\omega(\left(\begin{array}{c} a_1\\ a_2 \end{array} \right), \left(\begin{array}{c} b_1\\ b_2 \end{array} \right))=a_2^*b_1-a_1^*b_2.
$$
\end{defi}
Let us denote
$$
J=\left( \begin{array}{cc} 0 & 1 \\ -1 & 0 \end{array}\right) \in M_2(\a).
$$
Then, it is apparent that
$$
\omega(\left(\begin{array}{c} a_1\\ a_2 \end{array} \right),\left(\begin{array}{c} b_1\\ b_2 \end{array} \right))=<J\left(\begin{array}{c} a_1\\ a_2 \end{array} \right),\left(\begin{array}{c} b_1\\ b_2 \end{array} \right)>.
$$

We shall denote by $\tilde{a}, \tilde{b}, \tilde{c}$, etc. the elements of $M_2(\a)$. Let us denote by $Gl_2(\a)$ the group of invertible elements and by $\u_2(\a)$ the  group of unitary elements in $M_2(\a)$.

We shall use the selfadjoint reflection $\rho_H$ (i.e., $\rho_H^2=1$, $\rho_H^*=\rho_H$, ),  
$$
\rho_H=-iJ=\left( \begin{array}{cc} 0 & -i  \\ i & 0 \end{array}\right).
$$
and the  form $\theta_H$
$$
 \theta_H(\left(\begin{array}{c} x_1\\ x_2 \end{array} \right),\left(\begin{array}{c} y_1\\ y_2 \end{array} \right))=<\rho_H\left(\begin{array}{c} x_1\\ x_2 \end{array} \right),\left(\begin{array}{c} y_1\\ y_2 \end{array} \right)>=\frac{1}{i} (x_1^*y_2-x_2^*y_1),
$$
i.e., $\theta_H=i \omega$.

 The following group, which is the group of invertible matrices in $M_2(\a)$ which preserve the form $\theta_H$  (equivalently, the form $\omega$), will play an important role in this study:
\begin{equation}\label{simplectico}
\spa=\{\tilde{a}\in Gl_2(\a): \theta_H(\tilde{a}\left(\begin{array}{c} b_1\\ b_2 \end{array} \right),\tilde{a}\left(\begin{array}{c} c_1\\ c_2 \end{array} \right))=\theta_H(\left(\begin{array}{c} b_1\\ b_2 \end{array} \right),\left(\begin{array}{c} c_1\\ c_2 \end{array} \right)), 
$$
$$
 \hbox{ for all } \left(\begin{array}{c} b_1\\ b_2 \end{array} \right),\left(\begin{array}{c} c_1\\ c_2 \end{array} \right)\in \a^2\}.
\end{equation}
Clearly,
$$
\tilde{a}\in\spa \ \hbox{  if and only if } \ \rho_H\tilde{a}^*\rho_H=\tilde{a}^{-1}.
$$

\begin{defi}
Let $\k_H\subset\a^2$ be the set 
$$
\k_H=\{\left(\begin{array}{c} a_1\\ a_2 \end{array} \right)\in\a^2: \theta_H(\left(\begin{array}{c} a_1\\ a_2 \end{array} \right),\left(\begin{array}{c} a_1\\ a_2 \end{array} \right))=1, \hbox{ with } a_1\in G\},
$$
i.e., $\left(\begin{array}{c} a_1\\ a_2 \end{array} \right)\in\k_H$ if 
$\frac{1}{i}(a_1^*a_2-a_2^*a_1)= 2 Im(a_1^*a_2)=1$.
Notice that $a_2\in G$, automatically.
\end{defi}

We shall use  this hyperboloid $\k_H$ to understand the geometry of $\h$; it shall be a coordinate space for $\h$.

There is a natural fibration of $\k_H$ over $\h$:
$$
\varphi_H:\k_H\to\h, \varphi_H(\left(\begin{array}{c} x_1\\ x_2 \end{array} \right))=x_2x_1^{-1}.
$$

\section{The open unit disk model for $TG^+$}
We shall study an alternative model for $\h$ (i.e., for $TG^+$). The main reason to give this alternative version is that many computations  will be  clearer (and easier) in this model.

Let $\d$ be the open unit ball of $\a$:
$$
\d=\{z\in\a: z^*z<1\}=\{z\in\a: \|z\|<1\}.
$$

 Consider the selfadjoint reflection $\rho_D$:
$$
\rho_D=\left( \begin{array}{cc} 1 & 0 \\ 0 & -1 \end{array} \right) .
$$
The induced $\a$-valued indefinite inner product on $\a^2$
$$
\theta_D(\left(\begin{array}{c} x_1\\ x_2 \end{array} \right),\left(\begin{array}{c} y_1\\ y_2 \end{array} \right))=<\rho_D \left(\begin{array}{c} x_1\\ x_2 \end{array} \right),\left(\begin{array}{c} y_1\\ y_2 \end{array} \right)>=  x_1^*y_1-x_2^*y_2.
$$
The group $\u(\theta_D)$ of elements in $Gl_2(\a)$ which preserve the form $\theta_D$:
$$
\u(\theta_D)=\{\tilde{a}\in Gl_2(\a): \theta_D(\tilde{a }\left(\begin{array}{c} x_1\\ x_2 \end{array} \right), \tilde{a}\left(\begin{array}{c} y_1\\ y_2 \end{array} \right))=\theta_D(\left(\begin{array}{c} x_1\\ x_2 \end{array} \right),\left(\begin{array}{c} y_1\\ y_2 \end{array} \right))\}.
$$
Equivalently, $\tilde{a}\in \u(\theta_D)$  if $\rho_D\tilde{a}^*\rho_D=\tilde{a}^{-1}$.

Before proceeding any further, note that $\rho_D$ and $\rho_H$ are conjugate via the unitary element 
\begin{equation}\label{U}
U=\frac{1}{\sqrt{2}}\left( \begin{array}{cc} 1 & 1 \\ i & -i \end{array} \right) ,
\end{equation}
i.e., 
\begin{equation}\label{rhotheta}
U\rho_DU^*=\rho_H.
\end{equation}
  
Therefore, the groups $\spa$ and $\spb$ are conjugate, and the  properties and features of $\u(\theta_H)$ are translated to $\u(\theta_D)$. 

The sphere $\k_D$
$$
\k_D:=\{\left(\begin{array}{c} x_1\\ x_2 \end{array} \right)\in\a^2: \theta_D(\left(\begin{array}{c} x_1\\ x_2 \end{array} \right),\left(\begin{array}{c} x_1\\ x_2 \end{array} \right))=1, x_1\in G\}
$$
$$
=\{\left(\begin{array}{c} x_1\\ x_2 \end{array} \right)\in\a^2: x_1^*x_1-x_2^*x_2=1 , x_1\in G\}.
$$
There is an analogous fibration of $\k_D$ over $\d$:
$$
\varphi_D:\k_D\to\d,\ \  \varphi_D(\left(\begin{array}{c} x_1\\ x_2 \end{array} \right))=x_2x_1^{-1}.
$$
In fact,  if $\left(\begin{array}{l} x_1 \\ x_2 \end{array}\right)\in\k_D$, then
$$
x_1^*x_1=1+x_2^*x_2,
$$
and it follows that
$$
1=(x_1^{-1})^*(1+x_2^*x_2)x_1^{-1}=(x_1x_1^*)^{-1}+(x_2x_1^{-1})^*x_2x_1^{-1} ;
$$
thus
$$
(x_2x_1^{-1})^*x_2x_1^{-1}<1,
$$
i.e., $x_2x_1^{-1}\in\d$
\begin{lem}\label{Uactua}
$\left(\begin{array}{c} x_1\\ x_2 \end{array} \right)\in\k_D$ if and only if $U\left(\begin{array}{c} x_1\\ x_2 \end{array} \right)\in\k_H$. 
\end{lem}
\begin{proof}
The fact that $U\rho_DU^*=\rho_H$ implies that the equation $\theta_D(\left(\begin{array}{c} x_1\\ x_2 \end{array} \right),\left(\begin{array}{c} x_1\\ x_2 \end{array} \right))=1$ is equivalent to $\theta_H(U\left(\begin{array}{c} x_1\\ x_2 \end{array} \right),U\left(\begin{array}{c} x_1\\ x_2 \end{array} \right))=1$. 

Let us check the invertibility conditions. If $\left(\begin{array}{c} x_1\\ x_2 \end{array} \right)\in\k_D$, then $x_1\in G$. We must verify that the first coordinate of $U\left(\begin{array}{c} x_1\\ x_2 \end{array} \right)$ is invertible. Equivalently, that $1+x_2x_1^{-1}=(x_1+x_2)x_1^{-1}\in G$. Since $\left(\begin{array}{c} x_1\\ x_2 \end{array} \right)\in\k_D$, it follows that  $x_2x_1^{-1}\in\d$, and thus $\|x_2x_1^{-1}\|<1$. Then  $1+x_2x_1^{-1}$ is invertible.

Conversely, suppose that $\left(\begin{array}{c} z_1\\ z_2 \end{array} \right)=U\left(\begin{array}{c} x_1\\ x_2 \end{array} \right)\in\k_H$. Then $z_1, z_2$ are invertible, and we must check that $z_1-iz_2$ is also invertible. The fact that $\left(\begin{array}{c} z_1\\ z_2 \end{array} \right)\in\k_H$ implies that $Im(z_2z_1^{-1})>0$. Therefore a straightforward computation shows that  $-i\notin \sigma(z_2z_1^{-1})$. Then $z_2z_1^{-1}+i=i(z_1-iz_2)z_1^{-1}\in G$, i.e.,  $z_1-iz_2\in G$
\end{proof}

\section{The groups $\spa$ and $\spb$}
We shall describe in the next subsections the basic properties shared by $\spa$ and $\spb$. Some computations are easier or more natural in one of the two presentations of these isomorphic groups.

\subsection{The unitary  group $\spa$ of the form $\theta_H$}

In order to study $\spa$, we consider the following subgroups
\begin{defi}
Let $\b$ be the group of elements $\tilde{b}\in\spa$ which are of the form
$$
\tilde{b}=\left( \begin{array}{cc} b_{11} & 0 \\  b_{21} & b_{22} \end{array}\right)
$$ 
and $\t$ consisting of $\tilde{c}$,
$$
 \tilde{c}=\left( \begin{array}{cc} 1 & \tau \\  0 & 1 \end{array}\right),
$$ 
with $\tau^*=\tau$.
\end{defi}
It is apparent that $\b$, which we will call the {\it Borel subgroup} of $\spa$, is indeed a group. Also it is clear that $\t$ is a group. Note that
$$
J\tilde{c}^*J^{-1}=\left( \begin{array}{cc} 0 & 1 \\ -1 & 0 \end{array}\right)\left( \begin{array}{cc} 1 &  0 \\ \tau & 1 \end{array}\right)\left( \begin{array}{cc} 0 & -1 \\ 1 & 0 \end{array}\right)=\left( \begin{array}{cc} 1 & -\tau \\ 0 & 1 \end{array}\right)=\tilde{c}^{-1},
$$
i.e., $\tilde{c}\in\spa$.

We shall see that these groups are complemented Banach-Lie subgroups of $Gl_2(\a)$. 

An elementary computation shows that $\spa$ is closed under the involution of $M_2(\a)$: if $\tilde{a}\in\spa$ then $\tilde{a}^*\in\spa$. It follows that the factors of the polar decomposition of $\tilde{a}=\tilde{u}|\tilde{a}|$ remains inside $\spa$: $\tilde{u},|\tilde{a}|\in\spa$. It suffices to show that $|\tilde{a}|\in\spa$.
Clearly $|\tilde{a}|^2=\tilde{a}^*\tilde{a}\in\spa$. Note that both $J|\tilde{a}|J^{-1}$ and $|\tilde{a}|^{-1}$ are positive elements with the same square:
$$
(J|\tilde{a}|J^{-1})^2=J|\tilde{a}|^2J^{-1}=J\tilde{a}^*\tilde{a}J^{-1}=(\tilde{a}^*\tilde{a})^{-1}=|\tilde{a}|^{-2}.
$$
Then $J|\tilde{a}|J^{-1}=|\tilde{a}|^{-1}$.

The same is true for the other polar decomposition $\tilde{a}=|\tilde{a}^*|\tilde{w}$.

Note also the elementary fact that a unitary element $\tilde{u}$ belongs to $\spa$ if and only if it commutes with $J$. Similarly, a positive element $\tilde{b}\in\spa$ if and only if $J\tilde{b}J^{-1}=\tilde{b}^{-1}$.

We shall denote by $\uspa$ the (subgroup of) unitary elements of $\spa$, and by $\spa^+$  the set of positive elements in $\spa$.

\begin{prop}
The group $\spa$ is a $C^\infty$  Banach-Lie group, and a complemented submanifold of $M_2(\a)$. 
\end{prop}
\begin{proof}
Let us exhibit a local chart for $1\in\spa$. Denote by $M_2(\a)_s$ and $M_2(\a)_{as}$ the spaces of selfadjoint and anti-selfadjoint  elements of $M_2(\a)$. Consider the space
\begin{equation}\label{horizontalH}
\x=\x_{as}\oplus\x_{s}=\{\tilde{\beta}\in M_2(\a)_{as}: \tilde{\beta}J=J\tilde{\beta}\}\oplus \{\tilde{\gamma}\in M_2(\a)_s: \tilde{\gamma}J=-J\tilde{\gamma}\}.
\end{equation}
Elements $X\in\x$ are of the form $X=\tilde{\beta}+\tilde{\gamma}$,
$$
X=\left( \begin{array}{cc} \beta_{11} &  \beta_{12}  \\  -\beta_{12} & -\beta_{11} \end{array}\right) + \left( \begin{array}{cc} \gamma_{11} &  \gamma_{12}  \\  \gamma_{12} & -\gamma_{11} \end{array}\right)
$$
with $\beta_{11}^*=-\beta_{11}$, and all other entries selfadjoint.
Consider the map
$$
\e: \x\to\spa , \ \ \e(\tilde{\beta}+\tilde{\gamma})=e^{\tilde{\beta}}e^{\tilde{\gamma}}.
$$
Note that $e^{\tilde{\beta}}$ is a unitary element which commutes with $J$, i.e., $e^{\tilde{\beta}}\in\uspa$. The element $e^{\tilde{\gamma}}$ is a positive invertible element of $M_2(\a)$; the fact that the exponent $\tilde{\gamma}$ anticommutes with $J$ means that $e^{\tilde{\gamma}}J=Je^{-\tilde{\gamma}}$, i.e., $e^{\tilde{\gamma}}\in\spa^+$, and thus $\e$ is well defined. If one restricts $\e$ to
$$
\v=\{\tilde{\beta}\in\x_{as}: \|\tilde{\beta}\|<\pi\}\oplus\x_s,
$$
then $\e|_\v$ is a homeomorphism onto 
$$
\w=\{\tilde{a}\in\spa: \|\tilde{u}-1\|<2, \hbox{ where } \tilde{u}=\tilde{a}|\tilde{a}|^{-1}\}.
$$
Clearly, $\v$ and $\w$ are open sets of $\x$ and $\spa$, respectively. If $\tilde{a}\in\w$, the fact that   $\|\tilde{u}-1\|<2$ implies that $\tilde{u}=e^{\tilde{\delta}}$ for a unique $\tilde{\delta}\in M_2(\a)_{as}$ with $\|\tilde{\delta}\|<\pi$ (and thus $\tilde{\delta}$ is a series in powers of $u$). Since $\tilde{u}$ is unitary and belongs to $\spa$, it commutes with $J$. Then, its logarithm $\tilde{\delta}$ commutes with $J$, i.e., $\tilde{\delta}\in \x_1$. On the other hand, since $|\tilde{a}|$ is positive and invertible, it has a unique selfadjoint logarithm $\log(|\tilde{a}|)=\tilde{\epsilon}$. The fact that 
$J|\tilde{a}|J^{-1}=|\tilde{a}|^{-1}$ implies that  
$$
J\tilde{\epsilon}J^{-1}=J\log(|\tilde{a}|)J^{-1}=\log(J|\tilde{a}|J^{-1})=\log(|\tilde{a}|^{-1})=-\log(|\tilde{a}|)=-\tilde{\epsilon},
$$
i.e., $\tilde{\epsilon}\in\x_2$, and thus $\tilde{a}=\e(\delta,\epsilon)$. The inverse of $\e$ is 
$$
\e^{-1}:\w\to\v , \ \e^{-1}(\tilde{a})=\log(\tilde{a}|\tilde{a}|^{-1})+ \log(|\tilde{a}|)\in\x_1\oplus\x_2.
$$
It is apparent that both $\e$ and $\e^{-1}$ are $C^\infty$ maps.
The Banach space  $\x$ on which the neighbourhood $\w$ of $1$ in $\spa$ is modelled, is  complemented in $M_2(\a)$. Indeed, $\x$ can be also presented as
$$
\x=\{X\in M_2(\a): X=\left( \begin{array}{cc} x_{11} &  x_{12}  \\  x_{21} & -x_{11}^* \end{array}\right), \hbox{ with } x_{12}, x_{21}\in\a_s\}.
$$
A supplement for $\x$ is, for instance:
$$
\{\left( \begin{array}{cc} a &  b  \\  c & a^* \end{array}\right): b^*=-b, \  c^*=-c\}.
$$
Charts around other elements of $\spa$ are obtained by translation, using the left action of $\spa$ on itself. 
\end{proof}
\begin{rem}
The differential at the origin of the map $\e$ is the identity. It follows that the Banach-Lie algebra  $\mu(\theta_H)$ of $\spa$ coincides with $\x$: $X\in \mu(\theta_H)$   if 
$$
X=\left( \begin{array}{cc} x_{11} &  x_{12}  \\  x_{21} & -x_{11}^* \end{array}\right),
$$ 
where $x_{12}$ and $x_{21}$ are selfadjoint.
\end{rem}
\begin{prop}
$\b$ and $\t$ are Banach-Lie  subgroups of $\spa$, and complemented submanifolds of $M_2(\a)$. They generate an open and closed subgroup of $\spa$, which contains the connected component of the identity.
\end{prop}
\begin{proof}
First note that the diagonal entries of elements in $\b$ must be invertible elements in $\a$. Elementary matrix computations show that $\tilde{b}\in\b$ if and only if it is of the form
$$
\tilde{b}=\left( \begin{array}{cc} b & 0 \\ x & (b^*)^{-1} \end{array}\right),
$$
with $b^*xb^{-1}=x^*$, or, equivalently, $b^*x$ selfadjoint. Thus, $\b$ can be parametrized by 
$$
\mathbb{B}=\{(b,x): b \hbox{ invertible in } \a \hbox{ and } b^*x\in\a_s\}.
$$
This set $\mathbb{B}$ is globally diffeomorphic to the set 
$$
G\times\a_s=\{(a,y): a \hbox{ invertible in } \a , y=y^*\}.
$$
which is a complemented submanifold of $\a^2$. The diffeomorphism and its inverse are given by $(b,x)\mapsto (b,b^*x)$ and $(a,y)\mapsto (a, (a^*)^{-1}y)$. Thus, $\b$ is globally diffeomorphic to a submanifold of $\a^2$. 
Moreover, this diffeomorphism extends to an open subset of $M_2(\a)$:
$$
M_2(\a)=\{ \left( \begin{array}{cc} a & c \\ y+z & d \end{array}\right): (a,y)\in G\times \a_s,  \ z, c, d\in\a, \ z^*=-z\}\to M_2(\a),
$$
$$
\left( \begin{array}{cc} a & c \\ y+z & d \end{array}\right)\mapsto \left( \begin{array}{cc} a & c \\ (a^*)^{-1}y+z & (a^*)^{-1}+d \end{array}\right)
$$
which maps $G\times \a_s$ onto $\b$.

The Banach-Lie algebra $a_\b$ of $\b$ can be computed using this parametrization. If $b(t)\in G_\a$ and $x(t)\in\a$  are smooth curves such that $b^*(t)x(t)$ is selfadjoint,  $b(0)=1$, $\dot{b}(0)=y_{11}$, $x(0)=0$ and $\dot{x}(0)=y_{21}$, then
$$
\tilde{b}(t)=\left( \begin{array}{cc} b(t) & 0 \\ x(t) & (b^*(t))^{-1} \end{array}\right)
$$
is a smooth curve in $\b$ with $\tilde{b}(0)=1$ and $\dot{\tilde{b}}(0)=Y$,
$$
Y=\left( \begin{array}{cc} y_{11} & 0 \\ y_{21} & -y_{11}^* \end{array}\right),
$$
where $\frac{d}{dt}\{ b^*(t)x(t)\}=\dot{b^*}(t)x(t)+b^*(t)\dot{x}(t)\in\a_s$; in particular, at $t=0$, this implies that $y_{21}$ is selfadjoint. Thus, the Banach-Lie algebra $a_\b$ of $\b$  is
$$
a_\b=\{Y=\left( \begin{array}{cc} y_{11} & 0 \\ y_{21} & -y_{11}^* \end{array}\right): y_{21}^*=y_{21}\}.
$$

The subgroup $\t$ is parametrized by $\a_s$, and can be proved to be a submanifold of $M_2(\a)$ in a similar (simpler) fashion. Also, it is apparent that its Banach-Lie algebra $a_\t$ of $\t$  is 
$$
a_\t=\{Z=\left( \begin{array}{cc} 0 & z_{12} \\ 0 & 0 \end{array}\right): z_{12}^*=z_{12}\}.
$$

We claim that the subgroups $\b$ and $\t$ generate a closed and open subgroup of $\spa$. To this effect, note that the Banach-Lie algebras of these groups are in direct sum, and its sum is the Banach-Lie algebra $\mu(\theta_H)$ of $\uspa$:
$$
\mu(\theta_H)=a_\b\oplus a_\t.
$$
It follows that there is a neighbourhood $\w$ of $1$ in $\spa$ where any element is the product of elements in $\b$ and $\t$. Let $\tilde{a}_0\in\spa$ which is a product of elements in $\b$ and $\t$. Then 
$$
\u_0=\{\tilde{a}: \tilde{a}_0^{-1}\tilde{a}\in\u\}
$$
is an open neighbourhood of $\tilde{a}_0$ in $\spa$. It follows that the set of these products is an open subgroup. The relation 
$$
\tilde{a}\sim\tilde{b} \ \hbox{ if and only if } \  \tilde{a}^{-1}\tilde{b}\in \b\t
$$ 
is an equivalence relation. It follows that this subgroup is a union of connected components of $\spa$, containing  the connected component of the identity.
\end{proof}

\begin{rem}
If $\a$ is a von Neumann algebra, then $\spa$ is connected. Since $\spa^+$ is clearly connected (in fact, contractible), one needs to show that $\u_2(\a)\cap \{J\}'$ is connected. Since $J$ is anti-selfadjoint, it follows that $\{J\}'\subset M_2(\a)$ is a von Neumann algebra, and therefore $$\u_2(\a)\cap \{J\}'$$ is the unitary group of a von Neumann algebra, thus connected.
\end{rem}

\begin{teo}
The unitary part $$\u_2(\a)\cap \{J\}'$$ of $\spa$ is isomorphic to $\u_\a\times\u_\a$. The group $\Pi_0(\spa)$ of connected components of $\spa$, is isomorphic to $\Pi_0(\u_\a)\times \Pi_0(\u_\a)$.
\end{teo}
\begin{proof}

Consider the map
$$
\{J\}'\to \a , \ \ \left( \begin{array}{cc} a & b \\ -b & a \end{array}\right)\mapsto a+ib .
$$
This map is an injective C$^*$-homomorphism; thus, its restriction to $\u_2(\a)\cap \{J\}'$
$$
\Gamma: \uspa\to\u_\a
$$ is a group homomorphism. It is a retraction: the map $u\mapsto \left( \begin{array}{cc} u & 0 \\ 0 & u \end{array}\right)$ 
is a cross section for $\Gamma$ and a group homomorphism. 
By straightforward computations, the kernel of $\Gamma$ consist of matrices 
$$
\left( \begin{array}{cc} a & i(a-1) \\ -i(a-1) & a \end{array}\right)
$$
with $a^*a=aa^*=\frac12(a+a^*)$.  Then $a$ is a normal element, which is a zero of the continuous function $f(z)=|z|^2-Re(z)$. Therefore, the spectrum of $a$ is contained in the zero set of $f$, namely $\{z\in\mathbb{C}: |z-\frac12|=\frac12\}$.
The map $z\mapsto 2z-1$ sends this circle to the unit circle, and since it is a polynomial map, it sends elements $a$ as above onto normal elements with spectrum in the unit circle, i.e., unitary elements of $\a$. Conversely, if $u\in \u_\a$, elementary computations show that $a=\frac12(u+1)$ satisfies $a^*a=aa^*=\frac12(a+a^*)$. Moreover,  it is easy to verify that the map
$$
\ker \Gamma \to \u_\a , \ \left( \begin{array}{cc} a & i(a-1) \\ -i(a-1) & a \end{array}\right)\mapsto 2a-1
$$
is a group homomorphism, thus a bicontinuous isomorphism. 

Therefore, since $\Gamma$ splits, one has that (by means of an explicit isomorphism)
$$
\uspa \stackrel{\simeq}{\to} \u_\a\times \_\a.
$$
The polar decomposition induces the isomorphism between  $\Pi_0(\spa)$ and $\Pi_0(\u_U(\theta_H))$,  because the positive part $\u^+(\theta_H)$ is contractible.
\end{proof}

\begin{coro}
$\b$ and $\t$ generate $\spa$ if and only if $U_\a$ is connected.
\end{coro}

\section{The unitary group $\spb$ of the form $\theta_D$}

As remarked above, $\spb$ and $\spa$ are conjugate via the unitary operator $U$  given in (\ref{U}). Therefore the same properties proved for $\spa$ also hold for $\spb$. Thus, $\uspb$  is a Banach-Lie subgroup of $Gl_2(\a)$ and  it is closed under the polar decomposition: if $\tilde{a}\in\spb$ and $\tilde{a}=\tilde{u}|\tilde{a}|$ is its polar decomposition, then $\tilde{u}, |\tilde{a}|\in\spb$. The fact that the unitary part $\tilde{u}$ belongs to $\spb$ means that $\spb$ commutes with $\rho_D$. It is elementary that this implies that the matrix of $\tilde{u}$ is of the form
$$
\tilde{u}=\left( \begin{array}{cc} u_1 & 0 \\ 0 & u_2 \end{array} \right),
$$
with $u_1, u_2$  in $\u_\a$.
Denote by $\tilde{\lambda}=\log|\tilde{a}|$  the unique selfadjoint logarithm of the (positive invertible) element $|\tilde{a}|$.  The fact that $\rho|\tilde{a}|=|\tilde{a}|^{-1}\rho_D$, means that
$$
\rho_D\tilde{\lambda}=-\tilde{\lambda}\rho_D.
$$
On the other hand, positive elements $\tilde{r}\in\spb$ satisfy that $\rho_D\tilde{r}\rho_D\tilde{r}=1$. In particular, 
if 
$$
\tilde{r}=\left( \begin{array}{cc} r_{11} & r_{12} \\ r_{12}^*  & r_{22} \end{array} \right),
$$
then $r_{11}\ge 1$, because $r_{11}\ge 0$ and $r_{11}^2-r_{12}r_{12}^*=1$.  

The main issue in introducing this description of $\spb$ is the following result: 
\begin{teo}
$\spb$ acts on $\k_D$ by left multiplication: if $\tilde{a}\in\spb$ and $\left(\begin{array}{c} x_1\\ x_2 \end{array} \right)\in\k_D$, then $\tilde{a}\left(\begin{array}{c} x_1\\ x_2 \end{array} \right)\in\k_D$.
\end{teo}
\begin{proof}
Recall that $\left(\begin{array}{c} x_1\\ x_2 \end{array} \right)=(x_1,x_2)\in\k_D$ means that $\theta_D(\left(\begin{array}{c} x_1\\ x_2 \end{array} \right),\left(\begin{array}{c} x_1\\ x_2 \end{array} \right))=1$ and that $x_1\in G$. Since $\tilde{a}$ preserves $\theta_D$, it is clear that $\theta_D(\tilde{a}\left(\begin{array}{c} x_1\\ x_2 \end{array} \right),\tilde{a}\left(\begin{array}{c} x_1\\ x_2 \end{array} \right))=1$. We must show that the first coordinate of $\tilde{a}\left(\begin{array}{c} x_1\\ x_2 \end{array} \right)$ is invertible in $\a$. Clearly, it suffices to prove this fact separately for the unitary part $\tilde{u}$ and the absolute value $|\tilde{a}|$. The first assertion is clear:
$$
\tilde{u}\left(\begin{array}{c} x_1\\ x_2 \end{array} \right)=\left( \begin{array}{cc} u_1 & 0 \\ 0 & u_2 \end{array} \right)\left( \begin{array}{c} x_1  \\ x_2 \end{array} \right)=\left( \begin{array}{c} u_1x_1  \\ u_2 x_2 \end{array} \right),
$$
and $u_1x_1$ is invertible. 

For the second assertion, we claim that if $\tilde{r}$ is a positive element in $\spb$,
$$
\tilde{r}=\left( \begin{array}{cc} r_{11} & r_{12} \\ r_{12}^* & r_{22} \end{array} \right)
$$
which satisfies that $\|r_{12}\|<1$, then the first coordinate of  $\tilde{r}\left(\begin{array}{c} x_1\\ x_2 \end{array} \right)$ is invertible for any $\left(\begin{array}{c} x_1\\ x_2 \end{array} \right)\in\k_D$.
Indeed, the first coordinate of this product is $r_{11}x_1+r_{12}x_2$. Since $x_1$ is invertible, this sum is invertible if and only if $r_{11}+r_{12}x_2x_1^{-1}$  is invertible. Since $\|x_2x_1^{-1}\|<1$ and $\|r_{12}\|<1$, it follows that $\|r_{12}x_2x_1^{-1}\|<1$. Recall from above that $r_{11}\ge 1$. These facts imply that   $r_{11}+r_{12}x_2x_1^{-1}$  is invertible.

Note that for any $n\ge 1$, $|\tilde{a}|^{1/n}=e^{\frac1{n}\lambda}\in\spb$, because $\frac1{n}\lambda$ anti-commutes with $\rho_D$. Note also that $|\tilde{a}|^{1/n}\to 1$ as $n\to \infty$. Thus, there exists $n\ge 0$ such that $r=|\tilde{a}|^{1/n}$ satisfies that $\|r_{12}\|<1$. It follows from the above observation, that for any $\left(\begin{array}{c} x_1\\ x_2 \end{array} \right)\in\k_D$,  $|\tilde{a}|^{1/n}\left(\begin{array}{c} x_1\\ x_2 \end{array} \right)\in \k_D$. Then, inductively,
$$
|\tilde{a}| \left(\begin{array}{c} x_1\\ x_2 \end{array} \right)=|\tilde{a}|^{1/n}(\dots (|\tilde{a}|^{1/n}\left(\begin{array}{c} x_1\\ x_2 \end{array} \right))\in\k_D.
$$
\end{proof}
Therefore:
\begin{coro}
$\spa$ acts on $\k_H$ by left multiplication.
\end{coro}

\section{The actions of $\spb$ and $\spa$ on $\d$ and $\h$}

In this section we prove that $\u(\theta_H)$ acts in $TG^+$ (in fact, we prove this fact for the disk model  $\d$ of $TG^+$). First let us note that there is the natural inmersion $G^+\hookrightarrow TG^+$, $a\mapsto 0+ia$. $G^+$ is a homogeneous space of $G$, with the left action $g\cdot  a=(g^*)^{-1}ag^{-1}$. This action is the main feature in studying the geometry of $G^+$. We shall see that it can be regarded as a restriction of the action of $\uspa$ on $TG^+$ (see \cite{cprIEOT}, \cite{cprIJM}, \cite{cprILLINOIS}). Note that $G$ is a subgroup of $\uspa$, via the injective group homomorphism
$$
G\hookrightarrow \uspa \ , \ \ g\mapsto \left( \begin{array}{cc} g & 0 \\ 0 & (g^*)^{-1}\end{array} \right).
$$

In order to introduce the actions of $\spb$ and $\spa$ on $\d$ and $\h$, respectively, we  need the  maps:
$$
\varphi_D:\k_D\to \d \ , \ \varphi(x_1,x_2)=x_2x_1^{-1},
$$
and 
$$
\varphi_H:\k_H\to \h \ , \ \varphi(x_1,x_2)=x_2x_1^{-1}.
$$

\begin{defi}

\noindent

\begin{itemize}

\item
Given $z\in\d$, consider $(1,z)$ and compute $\theta_D((1,z),(1.z))=1-z^*z$. Since $\|z\|<1$, this element is positive and invertible. Then $(1-z^*z)^{-1/2},z(1-z^*z)^{-1/2})\in\k_D$.
If $\tilde{a}=\left( \begin{array}{cc} a_{11} & a_{12} \\ a_{21} & a_{22} \end{array} \right)\in\spb$, put
$$
\tilde{a}\cdot z:=\varphi_D(\tilde{a}\left( \begin{array}{l} (1-z^*z)^{-1/2} \\ z(1-z^*z)^{-1/2})\end{array}\right))=(a_{21}+a_{22}z)(a_{11}+a_{12}z)^{-1}.
$$
\item
Analogously, given $h\in\h$, $\theta_H((1,h),(1,h))=2\  Im(h)$. \\ Then
$(\frac{1}{\sqrt{2}}Im(h)^{-1/2}, \frac{1}{\sqrt{2}}h Im(h)^{-1/2})\in\k_H$. For $\tilde{a}=\left( \begin{array}{cc} a_{11} & a_{12} \\ a_{21} & a_{22} \end{array} \right)\in\spa$, put 
$$
\tilde{a}\cdot h:=\varphi_H(\tilde{a}\left( \begin{array}{l}  \frac{1}{\sqrt{2}} Im(h)^{-1/2} \\ \frac{1}{\sqrt{2}} h  Im(h)^{-1/2}\end{array}\right))= (a_{21} + a_{22}h)(a_{11}+ a_{12}h)^{-1}.
$$ 
\end{itemize}

\end{defi}

\begin{rem}
Both actions are well defined, and they are, indeed, left actions.

\end{rem}

Let us prove that these left actions are transitive. More specifically:
\begin{prop}
The action of the  subgroup $\b$ on $\h$ is transitive.
\end{prop}
\begin{proof}
Let $h\in\h$ and $\tilde{b}\in\b$,
$$
\tilde{b}=\left( \begin{array}{cc} b & 0 \\ x & (b^*)^{-1} \end{array} \right) ,
$$
with $b^*x$ selfadjoint. Then: 
$$
\tilde{b}\cdot h= xb^{-1}+(b^*)^{-1}hb^{-1}=(b^*)^{-1}(b^*x+h)b^{-1}.
$$ 
The map $x\mapsto (b^*)^{-1}xb^{-1}$ is a linear isomorphism in $\a$ which preserves positivity (and, thus, selfadjointness). Thus, 
$$
Im(\tilde{b}\cdot h)=(b^*)^{-1} Im(h) b^{-1}.
$$
 Fix $h\in\h$  and let $h'$ be another element in $\h$. Put  $z=Im(h)$ and $z'=Im(h')$; both are in $G^+$. The action of $G$ on $G^+$ is transitive, thus there exists $g\in G$ such that $$
z'=(g^*)^{-1}zg^{-1}.
$$
Put  $y=h'g-(g^*)^{-1}z$. Note that $g^*y=g^*h'g-z$ is selfadjoint:
$$
Im(g^*y)=g^* Im(h') g-z=g^* z' g-z=0.
$$
A direct computation shows that if 
$$
\tilde{g}=\left( \begin{array}{cc} g & 0 \\ y & (g^*)^{-1} \end{array} \right),
$$
then
$$
\tilde{g}\cdot h=h'.
$$
\end{proof}

\begin{rem}
As remarked, the unitary matrix $U$ maps $\k_D$ onto $\k_H$, and intertwines the groups $\uspb$ and $\uspa$: $\tilde{a}\in\uspa$ if and only if $U^*\tilde{a}U\in\uspb$. We shall see later (Remark \ref{HD}), that the Moebius transformation $\Gamma:\h\to\d$ induced by these transformations, maps $i\in\h$ to $0\in\d$. Let us denote by 
$\mathbb{I}_i^H$ the isotropy group (of the action of $\uspa$) of $i\in\h$:
$$
\mathbb{I}_i^H=\{\tilde{c}\in\uspa: \tilde{c}\cdot i=i\}.
$$
Accordingly, the isotropy group (of the action of $\uspb$) of $0\in\d$ is
$$
\mathbb{I}_0^D=\{\tilde{d}\in\uspb: \tilde{d}\cdot 0=0\}.
$$ 
The above facts imply that $U^*\mathbb{I}_i^HU=\mathbb{I}_0^D$.

On the other hand, note that $\tilde{d}\cdot 0=0$ if and only if $d_{21}=0$, since $\tilde{d}\in\uspb$, this implies that also $d_{12}=0$ and $d_{11},d_{22}\in\u_\a$, i.e.,
$$
\mathbb{I}_0^D=\{\left(\begin{array}{cc} u & 0 \\ 0 & u \end{array} \right) : u \in \u_\a\},
$$
and, therefore, $\mathbb{I}_i^H=U\mathbb{I}_0^DU^*=\mathbb{I}_0^D$.

\end{rem}

\begin{teo}\label{kh}

\noindent
\begin{enumerate}
\item
The subgroup $\b$ acts freely and transitively on $\k_H$. In particular,  $\spa$ acts transitively on $\k_H$.
\item
$\k_H$ is an $C^\infty$ submanifold of $\a^2$. 
\item
The restriction of $\varphi$ to $\k_H$,
$$
\varphi|_{\k_H}:\k_H\to\h
$$
is an $C^\infty$ epimorphism, with a global $C^\infty$ cross section.
\item
The group $\spa$ acts covariantly with respect to $\varphi$ (on the left): if $\tilde{a}\in\spa$ and 
$\left(\begin{array}{c} b_1\\ b_2 \end{array} \right)\in \k_H$,
$$
\varphi(\tilde{a}\left(\begin{array}{c} b_1\\ b_2 \end{array} \right))=\tilde{a}\cdot\varphi(\left(\begin{array}{c} a_1\\ a_2 \end{array} \right)).
$$
\end{enumerate}
\end{teo} 
\begin{proof}
The first assertion: consider the element $(1,i)\in\k_H$, and pick $\tilde{b}\in\b$,
$$
\tilde{b}=\left( \begin{array}{cc} b & 0 \\ x & (b^*)^{-1} \end{array} \right)
$$ 
with $b^*x$ selfadjoint. Then
$$
\tilde{b}\cdot (1,i)=(b, x+i(b^*)^{-1}).
$$
These pairs parametrize $\k_H$. Pick $\left(\begin{array}{c} x_1\\ x_2 \end{array} \right)=(x_1,x_2)\in\k_H$. Then $b=x_1$ and $x=x_2-i(x_1^*)^{-1}$ determine a matrix $\tilde{b}$ in $\b$: $x_1$ is invertible and a straightforward computation shows that the fact that $\left(\begin{array}{c} x_1\\ x_2 \end{array} \right)\in\k_H$ implies that $b^*x$ is selfadjoint. Clearly $\tilde{b}\cdot (1,i)=\left(\begin{array}{c} x_1\\ x_2 \end{array} \right)$, and $\tilde{b}$ is determined by this condition.

The second assertion: 
 the action of $\b$ provides a homeomorphism
 $$
 \sigma: \b\to \k_H , \sigma(\tilde{b})=\tilde{b}\cdot (1,i),
 $$
 with inverse
 $$
 \sigma^{-1} :\k_H\to \b, \sigma^{-1}(\left(\begin{array}{c} x_1\\ x_2 \end{array} \right))=\left( \begin{array}{cc} x_1 & 0 \\ x_2-i(x_1^*)^{-1} & (x_1^*)^{-1} \end{array} \right) .
$$
The map $\sigma^{-1}$ extends to $\tilde{\h}:=\{\left(\begin{array}{c} a_1\\ a_2 \end{array} \right)\in \a^2: Im(a_1^*a_2) , a_1\in G\}$, which is open in $\a^2$, and the map $\sigma$ extends to $M_2(\a)$. Therefore $\sigma$ provides a global  $C^\infty$ adapted chart for $\k_H$ (modelled in the manifold $\b$). 

The third assertion:
consider the map 
$$
\psi:\h\to \k_H \  , \ \ \psi(h)=\left( \begin{array}{c} \frac{1}{\sqrt{2}} Im(h)^{-1/2}\\  \frac{1}{\sqrt{2}} h\ Im(h)^{-1/2} \end{array}\right)=\frac{1}{\sqrt{2}}\left( \begin{array}{c} 1\\ h \end{array}\right)Im(h)^{-1/2}.
$$
By this description, it is apparent that $\psi$ takes values in $\tilde{\h}$, and moreover, after elementary computations (involving right multiplication by elements of of $G_\a$),
$$
i \omega(\psi(h))=Im( ( Im(h)^{-1/2})^* h\ Im(h)^{-1/2})=Im(h)^{-1/2}Im(h)Im(h)^{-1/2}=1.
$$
Finally, we get
$$
\varphi(\psi(h))=\varphi(\frac{1}{\sqrt{2}}Im(h)^{-1/2}, \frac{1}{\sqrt{2}}h Im(h)^{-1/2})=h.
$$ 

The fourth assertion:
$$
\tilde{a}\cdot\varphi(\left(\begin{array}{c} b_1\\ b_2 \end{array} \right))=\varphi(\tilde{a} \left(\begin{array}{c} 1 \\ b_2b_1^{-1} \end{array}\right)),
$$
using the invariance of $\varphi$ under the right action of $G_\a$, this equals
$$  
\varphi(\tilde{a}\left(\begin{array}{c} b_1 \\ b_2 \end{array}\right)b_1^{-1}))=\varphi(\tilde{a}\left(\begin{array}{c} b_1\\ b_2 \end{array} \right)).
$$
\end{proof}
\begin{coro}
$\k_D$ is a C$^\infty$-submanifold of $\a^2$
\end{coro}
\begin{proof}
In Theorem \ref{kh} it is shown that $\k_H$ is a C$^\infty$-submanifold of $\tilde{\h}$, which is an open subset of $\a^2$. Thus, $\k_D=U^*\k_H$ is also a submanifold of $\a^2$.
\end{proof}

Let us finish this section, by proving the claim made at the beginning of it, that the  action of $G$ in $G^+$ is the restriction of the action of $\uspa$ on $TG^+$ (using the model $\h$).
\begin{rem}
The injective group homomorphism $G\hookrightarrow \uspa$, described at the beginning of this section,  allows one to regard $g\in G$ as an element in $\uspa$, namely $\left( \begin{array}{cc} g & 0 \\ 0 & (g^*)^{-1} \end{array} \right)$. An element $a\in G^+$ lies in $\h$ as $0+ia$. Then 
$$
\left( \begin{array}{cc} g & 0 \\ 0 & (g^*)^{-1} \end{array} \right)\cdot (0+ia)=(g^*)^{-1}iag^{-1}=ig\cdot a,    
$$ 
which is the guise under which $g\cdot a\in G^+$ appears in $TG^+$.
\end{rem}
\section{The space $\d$ as decompositions of an indefinite form.}
The fact that $\rho_D$ is selfadjoint in $M_2(\a)$  and satisfies $\rho_D^2=1$ implies that the form $\theta_D$ induces a non degenerate $\a$-valued indefinite quadratic form in $\a^2$.
We shall consider the following set, which was studied in \cite{cpr} (Sections 3,4,6):
$$
\q_{\rho_D}=\{\epsilon\in M_2(\a): \epsilon^2=1  \hbox{ and } \rho_D\epsilon\in G^+\}.
$$ 
In particular, $\rho_D\epsilon$ is selfadjoint, which implies that
$$
\theta_D(\epsilon\left(\begin{array}{c} a_1\\ a_2 \end{array} \right),\left(\begin{array}{c} b_1\\ b_2 \end{array} \right))=\theta_D(\left(\begin{array}{c} a_1\\ a_2 \end{array} \right),\epsilon\left(\begin{array}{c} b_1\\ b_2 \end{array} \right)) \hbox{ for all } \left(\begin{array}{c} a_1\\ a_2 \end{array} \right), \left(\begin{array}{c} b_1\\ b_2 \end{array} \right)\in\a^2,
$$
i.e., $\epsilon$ is symmetric for the form $\theta_D$. 
The fact that $\epsilon^2=1$ implies that $\a^2$ is decomposed in two eigenspaces:
$$
\a^2_{+}=\{\left(\begin{array}{c} x_1\\ x_2 \end{array} \right)\in\a^2: \epsilon \left(\begin{array}{c} x_1\\ x_2 \end{array} \right)=\left(\begin{array}{c} x_1\\ x_2 \end{array} \right)\} \ , \  \ \a^2_{-}=\{\left(\begin{array}{c} y_1\\ y_2 \end{array} \right)\in\a^2: \epsilon \left(\begin{array}{c} y_1\\ y_2 \end{array} \right)=-\left(\begin{array}{c} y_1\\ y_2 \end{array} \right)\}.
$$
The fact that  $\rho_D\epsilon\in G^+$ means that the quadratic form induced by $\theta_D$ is positive definite in $\a^2_+$ and negative definite in $\a^2_-$, and that the eigenspaces are $\theta_D$-orthogonal. Conversely, any such decomposition of $\a^2$ induces a non selfadjoint reflection in $\q_{\rho_D}$. Therefore, it is appropriate to think of $\q_{\rho_D}$ as the set of positive-negative decompositions of the quadratic form (given by) $\theta_D$.

The set $\q_{\rho_D}$ is a submanifold of $M_2(\a)$. It has yet another important characterization in Section 3  of  \cite{cpr}: in the polar decomposition of $\epsilon$, the unitary part is precisely $\rho_D$:
$$
\epsilon=|\epsilon^*|\rho_D.
$$
Also, any nonselfadjoint reflection with this latter property belongs to $\q_{\rho_D}$.
Therefore, $\q_{\rho_D}$ is parametrized by a subset of $Gl_2(\a)^+$, the set  positive invertible elements in $M_2(\a)$. In particular, this endows $\q_{\rho_D}$ with an $C^\infty$ submanifold structure, with a rich metric  geometry of non-positive type.

The group $\u(\theta_D)$ acts transitively in $\q_{\rho_D}$,
$$
\tilde{g}\cdot \epsilon=g\epsilon g^{-1}.
$$

We shall see below that that $\q_{\rho_D}$ is naturally diffeomorphic to $\d$. In order to lighten the notation, when the elements $\left( \begin{array}{l} x_1 \\ x_2 \end{array} \right), \left( \begin{array}{l} y_1 \\ y_2 \end{array} \right)$, etc., appear as subindices,  let us denote them  by ${\bf x}, {\bf y}$, etc.

\begin{defi}
Any element $\left(\begin{array}{c} x_1\\ x_2 \end{array} \right)\in\k_D$ defines a (modular) rank one (non selfadjoint) projection in $M_2(\a)$: 
$$
p_{\bf x}=\left(\begin{array}{c} x_1\\ x_2 \end{array} \right) \left(\begin{array}{c} x_1\\ x_2 \end{array} \right)^*\rho_D=\left( \begin{array}{c} x_1 \\ x_2 \end{array} \right) \left( \begin{array}{cc} x_1^* & x_2^* \end{array} \right) \left( \begin{array}{cc}1 & 0 \\ 0  & -1 \end{array} \right)= \left( \begin{array}{cc} x_1x_1^* &  -x_1x_2^* \\ x_2x_1^* &  -x_2x_2^* \end{array} \right).
$$
Or, equivalently, 
$$
p_{\bf x}(\left(\begin{array}{c} a_1\\ a_2 \end{array} \right))=\left(\begin{array}{c} x_1\\ x_2 \end{array} \right)\  \theta_D(\left(\begin{array}{c} a_1\\ a_2 \end{array} \right),\left(\begin{array}{c} x_1\\ x_2 \end{array} \right))=\left(\begin{array}{c} x_1\\ x_2 \end{array} \right)<\rho_D\left(\begin{array}{c} a_1\\ a_2 \end{array} \right) , \left(\begin{array}{c} x_1\\ x_2 \end{array} \right)>.
$$
 With this description, since $\theta_D(\left(\begin{array}{c} x_1\\ x_2 \end{array} \right),\left(\begin{array}{c} x_1\\ x_2 \end{array} \right))=1$, it it clear that $p_{\bf x}$ is a projection, which is  $\theta_D$-symmetric.

Consider the following maps:
\begin{equation}\label{Phi_D}
\Phi_D: \d\to \q_{\rho_D} \  , \ \Phi_D(h)=2p_{\bf x}-1,
\end{equation}
where $\left(\begin{array}{c} x_1\\ x_2 \end{array} \right)\in\k_D$  satisfies that $\varphi(\left(\begin{array}{c} x_1\\ x_2 \end{array} \right))=h$,  
and
\begin{equation}\label{phitilde}
\tilde{\varphi}:\k_D\to\q_{\rho_D}\  , \ \ \tilde{\varphi}(\left(\begin{array}{c} x_1\\ x_2 \end{array} \right))=2p_{\bf x}-1.
\end{equation}
\end{defi}
\begin{lem}
The map $\Phi_D$ is well defined, $C^\infty$ and equivariant with respect to the actions of $\u(\theta_D)$. 
\end{lem}
\begin{proof}
Suppose $\left(\begin{array}{c} x_1\\ x_2 \end{array} \right), \left(\begin{array}{c} y_1\\ y_2 \end{array} \right)\in\k_D$ such that $x_2x_1^{-1}=y_2y_1^{-1}$, i.e., $\left(\begin{array}{c} y_1\\ y_2 \end{array} \right)=\left(\begin{array}{c} x_1\\ x_2 \end{array} \right)\cdot g$, for $g=x_1^{-1}y_1\in G$. Note that 
$$
1=y_1^*y_1-y_2^*y_2=(x_1g)^*x_1g-(x_2g)^*x_2g=g^*(x_1^*x_1-x_2^*x_2)g= g^*g,
$$
i.e., $g\in\u_\a$. Then
$$
p_{\bf y}=\left(\begin{array}{c} y_1\\ y_2 \end{array} \right)\left(\begin{array}{c} y_1\\ y_2 \end{array} \right)^*\rho_D=\left(\begin{array}{c} x_1\\ x_2 \end{array} \right)gg^*\left(\begin{array}{c} x_1\\ x_2 \end{array} \right)^*\rho_D=\left(\begin{array}{c} x_1\\ x_2 \end{array} \right)\left(\begin{array}{c} x_1\\ x_2 \end{array} \right)^*\rho_D
$$
$$
=p_{\bf x}.
$$
Let us prove that the reflection $2p_{\bf x}-1$ belongs to $\q_{\rho_D}$. It is  symmetric for the form $\theta_D$. Clearly, $\rho_D(2p_{\bf x}-1)$ is invertible, and it is non negative if and only if 
$$
(2p_{\bf x}-1))\rho_D=\rho_D(\rho_D(2p_{\bf x}-1))\rho_D\ge 0.
$$
Explicitly,
$$
(2 p_{\bf x}-1)\rho_D=\left( \begin{array}{cc} x_1x_1^*-1  & x_1x_2^* \\  x_2x_1^*  & x_2x_2^*+1 \end{array} \right).
$$
Put $\gamma=x_1x_2^*$. Since $\left(\begin{array}{c} x_1\\ x_2 \end{array} \right)\in\k_D$, one has that $x_1^*x_1=1+x_2^*x_2$. Then  
$$
\gamma^*\gamma= x_2x_1^*x_1x_2^*=x_2(1+x_2^*x_2)x_2^*=x_2x_2^*+(x_2x_2^*)^2,
$$
thus $\gamma^*\gamma+1=(x_2x_2^*+1)^2$, i.e., 
$x_2^*x_2+1=(\gamma^*\gamma+1)^{1/2}$. Similarly, we get $x_1x_1^*-1=(\gamma\gamma^*+1)^{1/2}$ (using now that $x_1x_1>1$, because $x_1^*x_1>1$ and $x_1$ is invertible) .
Then
$$
(2 p_{\bf x}-1)\rho_D=\left( \begin{array}{cc} (\gamma\gamma^*+1)^{1/2} & \gamma \\ \gamma^* & (\gamma^*\gamma+1)^{1/2} \end{array} \right).
$$
In order to prove that this matrix is invertible, denote 
$$
m=\left(\begin{array}{cc} 0 & \gamma \\ \gamma^* & 0 \end{array} \right).
$$
Clearly $m$ is selfadjoint and 
$$
(2 p_{\bf x}-1)\rho_D=(1+m^2)^{1/2}+m\ge 0,
$$
because the real function $f(t)=(1+t^2)^{1/2}+t\ge 0$ for all $t\in\mathbb{R}$.

The fact that $\Phi_D$ is $C^\infty$ follows from a standard argument in fibrations: clearly, the formula that defines $\Phi_D$ in terms of coordinates in $\k_D$ is $C^\infty$, therefore, using  $C^\infty$ local cross sections for the fibration $\hat{\varphi}:\k_D\to \d$, one obtains that $\Phi_D$ is $C^\infty$.

Finally, if $\tilde{g}\in\u(\theta_D)$ (i.e., $\tilde{g}^*\rho_D=\rho_D\tilde{g}^{-1}$),
$$
p_{\tilde{g}\cdot {\bf x}}=\tilde{g}\left(\begin{array}{c} x_1\\ x_2 \end{array} \right)<\rho_D\ \cdot \ , \tilde{g}\left(\begin{array}{c} x_1\\ x_2 \end{array} \right)>=\tilde{g}\left(\begin{array}{c} x_1\\ x_2 \end{array} \right)<\tilde{g}^*\rho_D\ \cdot \ , \left(\begin{array}{c} x_1\\ x_2 \end{array} \right)>
$$
$$
=\tilde{g}\left(\begin{array}{c} x_1\\ x_2 \end{array} \right)<\rho_D\tilde{g}^{-1}\ \cdot \ , \left(\begin{array}{c} x_1\\ x_2 \end{array} \right)>
=\tilde{g} p_{\bf x} \tilde{g}^{-1},
$$
which means that $\Phi_D$ is $\u(\theta_D)$-equivariant: $\Phi_D(\tilde{g}\left(\begin{array}{c} x_1\\ x_2 \end{array} \right))=\tilde{g}\Phi_D(\left(\begin{array}{c} x_1\\ x_2 \end{array} \right))\tilde{g}^{-1}$.
\end{proof}

In the next theorem we summarize several results about the maps considered in the following diagram of homogeneous spaces of the group $\u(\theta_D)$:
\begin{equation}\label{diagrama}
\xymatrix{
& \k_D \ar[ld]_{\hat{\varphi}}\ar[rd]^{\tilde{\varphi}} \\
\d \ar[rr]^{\Phi_D}
&& \q_{\rho_D}
}
\end{equation}

\begin{teo}
 The diagram (\ref{diagrama}) commutes.
The maps $\hat{\vp}:\k_D\to \d$ and $\tilde{\vp}:\k_D\to \q_{\rho_D}$ are $C^\infty$ submersions. The map
$\Phi_D:\d\to\q_{\rho_D}$ is a $C^\infty$-diffeomorphism. All maps are equivariant under the action of $\u(\theta_D)$.
\end{teo}
\begin{proof}
Clearly,
$$
\Phi_D\circ\hat{\varphi}(\left(\begin{array}{c} x_1\\ x_2 \end{array} \right))=p_{\bf x}=\tilde{\varphi}(\left(\begin{array}{c} x_1\\ x_2 \end{array} \right)),
$$
thus the diagram commutes. Let us prove that every reflection $\epsilon$ in $M_2(\a)$ which is $\theta_D$ symmetric and such that $\rho_D\epsilon$ (equivalently, $\epsilon\rho_D$ is positive) must be of the form $\epsilon=2p_{\bf x}-1$ for $x\in\k_D$. That is,
 given $\epsilon$ of the form
 $$
 \epsilon=\left( \begin{array}{cc} \epsilon_{11} & \epsilon_{12} \\ -\epsilon_{12}^*  & \epsilon_{22} \end{array} \right)
 $$
 with $\epsilon_{ii}^*=\epsilon_{ii}$ such that
 $$
 0\le \epsilon\rho_D=\left( \begin{array}{cc} \epsilon_{11} & -\epsilon_{12} \\ -\epsilon_{12}^*  & -\epsilon_{22} \end{array} \right),
 $$
 there exists $\left(\begin{array}{c} x_1\\ x_2 \end{array} \right)\in\k_D$ such that 
 $$
 \left( \begin{array}{cc} \epsilon_{11} & \epsilon_{12} \\ -\epsilon_{12}^*  & \epsilon_{22} \end{array} \right)=\left( \begin{array}{cc} x_1x_1^*-1 & -x_1x_2^* \\ x_2x_1^*  & -x_2x_2^*-1 \end{array} \right).
 $$
 We shall look for $\left(\begin{array}{c} x_1\\ x_2 \end{array} \right)$ with $x_1>0$. Note that for any $\left(\begin{array}{c} x_1\\ x_2 \end{array} \right)\in\k_D$ there exists $\left(\begin{array}{c} x'_1\\ x'_2 \end{array} \right)$ with $x_2'(x_1')^{-1}=x_2x_1^{-1}$ and $x_1'>0$. Indeed, since $x_1$ is invertible, put $x_1=|x_1^*|u$ the polar decomposition of $x_1^*$, and take $x_1'=|x_1^*|$. It is not hard to see that such $x_1'$ is unique.
 
The fact that $\epsilon^2=1$ means  that
$$
\left\{ \begin{array}{l} \epsilon_{11}^2-\epsilon_{12}\epsilon_{12}^*=1 \\ \epsilon_{11}\epsilon_{12}+\epsilon_{12}\epsilon_{22}=0 \\ -\epsilon_{12}^*\epsilon_{12}+\epsilon_{22}^2=1 \end{array} \right. .
$$

Thus, $x_1=(1+\epsilon_{11})^{1/2}$ and $x_2=-\epsilon_{12}^*(1+\epsilon_{11})^{-1/2}$.
We must check that $x_2x_1^*=-\epsilon_{12}^*$, which is apparent, and that
$-x_2x_2^*-1=\epsilon_{22}$. Indeed, by the above relation on $e_{ij}$,
$$
-x_2x_2^*-1=\epsilon_{12}^*(1+\epsilon_{11})^{-1}\epsilon_{12}-1=\epsilon_{12}^*(1+(\epsilon_{12}\epsilon_{12}^*+1)^{1/2})^{-1}\epsilon_{12}-1.
$$
Since $\epsilon_{12}^*(\epsilon_{12}\epsilon_{12}^*)^n=(\epsilon_{12}^*\epsilon_{12})^n\epsilon_{12}^*$, then for any continuous function $f:[0,+\infty)\to \mathbb{C}$, 
$$
\epsilon_{12}^*f(\epsilon_{12}\epsilon_{12}^*)=f(\epsilon_{12}^*\epsilon_{12}) \epsilon_{12}^*.
$$
Thus, again using the relations on $\epsilon_{ij}$,
$$-x_2x_2^*-1=(1+(\epsilon_{12}^*\epsilon_{12}+1)^{1/2})^{-1}\epsilon_{12}^*\epsilon_{12}-1=(1+\epsilon_{22})^{-1}(\epsilon_{22}^2-1)-1=\epsilon_{22}.
$$ 
Next, we must check that $\left(\begin{array}{c} x_1\\ x_2 \end{array} \right)\in\k_D$. Clearly $x_1=(1+\epsilon_{11})^{1/2}$ is invertible (and positive).
$$
x_1x_1^*-x_2x_2^*=1+\epsilon_{11}-(1+\epsilon_{11})^{-1}\epsilon_{12}\epsilon_{-12}^*(1+\epsilon_{11})^{-1/2}.
$$
Since $-\epsilon_{12}\epsilon_{12}^*=(1+\epsilon_{11})^2$, this equals
$$
1+\epsilon_{11}-(-1+\epsilon_{11})=2.
$$
The formula $\left(\begin{array}{c} x_1\\ x_2 \end{array} \right)=\left( \begin{array}{c} (1+\epsilon_{11})^{1/2} \\ -\epsilon_{12}^*(1+\epsilon_{11})^{-1/2}\end{array} \right)$, regarded as a map $\q_{\rho_D}\to \k_D$, provides a global $C^\infty$ cross section for   $\tilde{\vp}$, proving that it is retraction, thus  a submersion. 

Let us exhibit the inverse of $\Phi_D$: since $\epsilon=2p_{\bf x}-1$ for a unique $\left(\begin{array}{c} x_1\\ x_2 \end{array} \right)\in\k_D$ with $x_1>0$ (as computed above, 
$$
\Phi_D^{-1}:\q_{\rho_D}\to \d \ , \ \ \Phi_D^{-1}(\epsilon)=x_2x_1^{-1}=-\epsilon_{12}^*(1+\epsilon_{11})^{-1}.
$$
Clearly, it is a $C^\infty$ map.
\end{proof}
The map $\Phi_D$ was computed using coordinates in $\k_D$.  It can be also computed in terms of $z\in\d$. Analogously as  $\tilde{\vp}$, $\hat{\vp}$ has also a global cross section given by the unique element in each fiber with positive first coordinate. Namely,
$$
\delta:\d\to\k_D \ , \ \ \delta(z)=\left( \begin{array}{c} 1 \\ z \end{array} \right) (1-z^*z)^{-1/2}.
$$
Then 
$$
\Phi_D(z)=2p_{\delta(z)}-1= 2\left( \begin{array}{c} 1 \\ z \end{array} \right)(1-z^*z)^{-1} \left( \begin{array}{cc} 1 &  z^* \end{array} \right) \rho_D -1
$$
$$
=\left( \begin{array}{cc} 2(1-z^*z)^{-1}- 1  & - 2(1-z^*z)^{-1}z^*\\ 2z(1-z^*z)^{-1} & -2z(1-z^*z)^{-1}z^* -1 \end{array} \right).
$$
\begin{rem}\label{HD}

There is an analogous diagram as (\ref{diagrama}) for the space $\h$:
\begin{equation}\label{diagrama H}
\xymatrix{
& \k_H \ar[ld]_{\varphi}\ar[rd]^{\bar{\varphi}} \\
\h \ar[rr]^{\Phi_H}
&& \q_{\rho_H}
}
\end{equation}

Recall the unitary operator $U$ which intertwines $\rho_D$ and $\rho_H$: $U\rho_DU^*=\rho_H$. We saw (Lemma \ref{Uactua}) that $U$ maps $\k_D$ onto $\k_H$. Also,  it is clear that 
$$
\epsilon\in\q_{\rho_H} \hbox{ if and only if } U^*\epsilon U\in\q_{\rho_D}.
$$
Also it is clear, by construction, that
\begin{equation}\label{equivaHD}
\Phi_H(h)=U \Phi_D(\gamma(h)) U^*.
\end{equation}
\end{rem}
Let us state another consequence obtained from the equivalence of diagrams (\ref{diagrama}) and (\ref{diagrama H}
\begin{teo}
The map $\Gamma:\h\to \d$, 
$$
\Gamma(h)=(1+ih)(1-ih)^{-1}
$$
is a (well defined) diffeomorphism with inverse $\Gamma^{-1}:\d\to\h$
$$
\Gamma^{-1}(z)=i(1-z)(1+z)^{-1}.
$$
\end{teo}
\begin{proof}
One passes from $\h$ to $\d$ with the cross section 
$$
h\mapsto \frac{1}{\sqrt{2}}\left( \begin{array}{c} 1 \\ h \end{array} \right) Im(h)^{-1/2}\in\k_H
$$
composed with left multiplication by $U^*$, followed by $\hat{\vp}$, namely,
$$
h\mapsto \frac{1}{\sqrt{2}}\left( \begin{array}{c} 1 \\ h \end{array} \right) Im(h)^{-1/2}\mapsto \frac{1}{2}\left( \begin{array}{cc} 1 & -i \\ 1 & i  \end{array} \right)\left( \begin{array}{c} 1 \\ h \end{array} \right) Im(h)^{-1/2}\mapsto (1+i h)(1-ih)^{-1}.
$$
Which means that $\Gamma:\h\to\d$ is well defined and smooth. Its inverse is computed analogously:
$$
z\mapsto \left( \begin{array}{l} 1 \\ z \end{array} \right)(1-z^*z)^{-1/2} \mapsto U \left( \begin{array}{l} 1 \\ z \end{array} \right)(1-z^*z)^{-1/2}\stackrel{\hat{\varphi}}{\to} (i-iz)(1+z)^{-1}.
$$
A straightforward computation shows that these maps are each other inverses.
\end{proof}

Let us finish this section by recalling the action of the group $\spa$.
 If  $\left(\begin{array}{c} x_1\\ x_2 \end{array} \right)\in\tilde{\h}$, then, since
$p_{\bf x}(\left(\begin{array}{c} y_1\\ y_2 \end{array} \right))=\left(\begin{array}{c} x_1\\ x_2 \end{array} \right) \theta_H(\left(\begin{array}{c} x_1\\ x_2 \end{array} \right), \left(\begin{array}{c} y_1\\ y_2 \end{array} \right) )$, one has that for any $\tilde{a}\in\spa$,
$$
 p_{\tilde{a}\cdot {\bf x}}=\tilde{a}.\left(\begin{array}{c} x_1\\ x_2 \end{array} \right) \theta_H(\tilde{a}.\left(\begin{array}{c} x_1\\ x_2 \end{array} \right),\cdot )=\tilde{a}.\left(\begin{array}{c} x_1\\ x_2 \end{array} \right) \theta_H(\left(\begin{array}{c} x_1\\ x_2 \end{array} \right),\tilde{a}^{-1} \cdot )=\tilde{a} p_{\bf x} \tilde{a}^{-1}.
$$
Then $2p_{\tilde{a}\cdot {\bf x}}-1=\tilde{a}(2p_{\bf x}-1)\tilde{a}^{-1}$. Therefore:
\begin{prop}\label{equivPhi}
If $h\in\h$ and $\tilde{g}\in\uspa$,
$$
\Phi_H(\tilde{g}\cdot h)=\tilde{g} \Phi_H(h)\tilde{g}^{-1},
$$
i.e., $\Phi_H$ is equivariant for the action of $\spa$. Therefore
\begin{equation}\label{gamma y U}
\Gamma(\tilde{g}\cdot h)=(U^*\tilde{g}U)\cdot\Gamma(h)
\end{equation}
\end{prop}

\section{The hyperbolic geometry of $\q_\rho$}
In previous papers \cite{cprIEOT}, \cite{cprIJM}, \cite{cprILLINOIS} the geometry of the set  of positive invertible elements of a $C^*$-algebra was studied. As is the case with the classical example of positive definite complex matrices (see for instance \cite{mostow}), this space, with the appropriate Finsler metric, behaves as a non-positively curved manifold. In \cite{prILLINOIS} it is proven that the set $\q_\rho$ embedds in the space of positive invertible elements, in a way that the main geometric features remain invariant. Let us briefly describe below  these constructions and  results. We shall use these results (and therefore also state them) for the case of the $C^*$-algebra $M_2(\a)$; the space of positive and invertible matrices shall be denoted by $Gl_2(\a)^+$. 
\begin{rem} (see \cite{cprIEOT})

$Gl_2(\a)^+$ is an open subset of $M_2(\a)_s$, so it has a natural differentiable structure. We consider in $Gl_2(\a)^+$ the following left action of $Gl_2(\a)$:
$$
\tilde{g}\cdot \tilde{a}= (\tilde{g}^*)^{-1}\tilde{a}\tilde{g}^{-1} , \ \ \tilde{g}\in Gl_2(\a), \tilde{a}\in Gl_2(\a)^+.
$$
This action is transitive. The isotropy subgrup of an element $\tilde{a}\in Gl_2(\a)^+$ is the group of $\tilde{a}$-unitary operators. Thus, the Banach-Lie isotropy algebra of $\tilde{a}$ is the space of $\tilde{a}$-anti-Hermitian elements of $M_2(\a)$: $\tilde{x}^*\tilde{a}+\tilde{a}\tilde{x}=0$. A natural complement for this space is the space of $\tilde{a}$-Hermitian elements: $\tilde{y}^*\tilde{a}-\tilde{a}\tilde{y}=0$. This decomposition is equivariant under the action of $Gl_2(\a)$, and induces a linear connection in $Gl_2(\a)^+$. The covariant derivative of this connection is given by 
$$
\frac{DY}{dt}=\frac{dY}{dt}-\frac12\{\dot{\gamma}\gamma^{-1}Y+Y\gamma^{-1}\dot{\gamma}\},
$$
where $Y(t)$ is a tangent field along the curve $\gamma(t)$ in $Gl_2(\a)^+$;  due to the trivial local structure of $Gl_2(\a)^+\subset M_2(\a)_s$, this simply means that $Y(t)$ is a curve of selfadjoint elements in $M_2(\a)$.  A geodesic is a curve $\gamma$ such that $\frac{D\dot{\gamma}}{dt}=0$. The geodesic $\gamma$ with $\gamma(0)=\tilde{a}$ and $\dot{\gamma}(0)=X$ is given by
$$
\gamma(t)=e^{\frac{t}2 X \tilde{a}^{-1}}\tilde{a} e^{\frac{t}2 X\tilde{a}^{-1}}.
$$
The exponential map $exp_{\tilde{a}}: M_2(\a)_{s}\to Gl_2(\a)^+$,
$$
exp_{\tilde{a}}(X)=e^{\frac12X\tilde{a}^{-1}}\tilde{a}e^{\frac12X\tilde{a}^{-1}},
$$
 is everywhere a diffeomorphism. Any pair $\tilde{a},\tilde{b}\in Gl_2(\a)^+$ is joined by a unique geodesic, which is
$$
\gamma_{\tilde{a},\tilde{b}}(t)=\tilde{a}^{1/2}(\tilde{a}^{-1/2}\tilde{b}\tilde{a}^{-1/2})^t\tilde{a}^{1/2}.
$$
\end{rem}
The space $Gl_2(\a)^+$ carries a Finsler metric. By this we mean a continuous distribution $Gl_2(\a)^+ \ni \tilde{a} \mapsto \|\ \|_{\tilde{a}}$ of norms defined in the corresponding tangent spaces $T(Gl_2(\a)^+)_{\tilde{a}}=M_2(\a)_s$. Notice that we do not require that this distribution be smooth, as in the finite dimensional setting (see for instance \cite{atkin1}, \cite{atkin2}, \cite{cprIEOT}, \cite{cpr}, for other examples of  Finsler metrics in the context of operator theory). If $\tilde{a}\in Gl_2(\a)^+$ and $X\in M_2(\a)_s$, put
\begin{equation}\label{metrica}
\|X\|_{\tilde{a}}=\|\tilde{a}^{-\frac12}X\tilde{a}^{-\frac12}\|.
\end{equation}
\begin{teo}{\rm  (\cite{cprIEOT}, Section 6)}
With the metric defined in (\ref{metrica}), the geodesics of the connection are globally minimal: for any  $\tilde{a},\tilde{b}\in Gl_2(\a)^+$,  the unique geodesic  $\gamma_{\tilde{a},\tilde{b}}$ of the connection has minimal length among all smooth curves in $Gl_2(\a)^+$ joining $\tilde{a}$ and $\tilde{b}$.
 
The (geodesic) distance between $\tilde{a}$ and $\tilde{b}$ can be computed:
$$
d_g(\tilde{a},\tilde{b})=\|\log(\tilde{a}^{-\frac12}\tilde{b}\tilde{a}^{-\frac12})\|,
$$
where $\log$ denotes the unique selfadjoint logarithm of a positive invertible element.
\end{teo}

	Moreover, the metric has the following property, which in Riemannian geometry is equivalent to non positive curvature, and is used as a definition of non positive curvature for metric length spaces (i.e., metric spaces with given short curves, see \cite{bridson}, \cite{burago}, \cite{gromov}). 

\begin{teo}{\rm  (\cite{cprILLINOIS})}

If  $\gamma(t), \delta(t)$ are two  geodesics in $G^+$,  then $f(t)=d_g(\gamma(t),\delta(t))$ is a  convex function.

In particular, if the geodesics start at the same point, i.e., $\gamma(0)=\delta(0)$, then
$$
d_g(\gamma(t),\delta(t))\le t\  d_g(\gamma(1),\delta(1)).
$$
\end{teo}

There is a natural embedding of $\q_\rho$ in $Gl_2^+(\a)$:
\begin{teo} \rm{ (\cite{prILLINOIS})}

The embedding of $\q_\rho$ in $Gl_2(\a)^+$ is given by \cite{prILLINOIS} 
$$
\q_\rho\hookrightarrow Gl_2^+(\a) , \ \epsilon\to \rho\epsilon.
$$
This embedding has the following properties
\begin{enumerate}
\item
Let $\epsilon_1, \epsilon_2\in\q_\rho$. Then the unique geodesic of $Gl_2(\a)^+$ joining $\rho\epsilon_1$ and $\rho\epsilon_2$ lies in (the image of) $\q_\rho$ (under the above embedding).
\item
$\q_\rho$ is an homogeneous space under the action of $\u(\theta_\rho)$, $Gl_2^+(\a)$ is an homogeneous space under the action of $Gl_2(\a)$. The embedding is equivariant for these actions. 
\end{enumerate}
\end{teo} 

Thus, if we endow $\q_\rho$ with the geometry induced by this embedding, it becomes a non positively curved metric length space.

\bigskip

We use these facts to translate  to $\h$ and $\d$  the metric structure of $\q_\rho$, by means of the equivariant diffeomorphisms   $\Phi_H$ and $\Phi_D$, respectively. 
\begin{defi}
For any $h_1,h_2\in\h$ and $z_1,z_2\in\d$,
$$
d_H(h_1,h_2)=d_g(\rho_H \Phi_H(h_1),\rho_H \Phi_H(h_2))
$$
and
$$
d_D(z_1,z_2)=d_g(\rho_D \Phi_D(z_1),\rho_D \Phi_D(z_2)).
$$
\end{defi}
Therefore one has:
\begin{coro}
Both $(\h,d_H)$ and $(\d,d_D)$ are non positively curved metric length spaces. In other words,
if $\delta_1,\delta_2$ are two geodesics in $\h$ (resp. $\d$), the function
$$
f(t)=d_H(\delta_1(t),\delta_2(t))
$$
(resp. $f(t)=d_D(\delta_1(t),\delta_2(t))$) are convex.

If, additionally,  $\delta_1(0)=\delta_2(0)$ then, for $t\in[0,1]$
\begin{equation}\label{cuerda y arco}
d_H(\delta_1(t),\delta_2(t)) \le t\   d_H(\delta_1(1),\delta_2(1))  \ {\rm ( } \hbox{ resp. }  d_D(\delta_1(t),\delta_2(t))\le  t \ d_D(\delta_1(1),\delta_2(1)) \ {\rm )} .
\end{equation}

The diffeomorphism
$$
\Gamma:\h\to\d
$$
is an isometry.
\end{coro}
\begin{proof}
Recall formula (\ref{Uactua}) in Remark (\ref{HD}): $\Phi_H(h)=U \Phi_D(\gamma(h)) U^*$. A straightforward computation shows that this implies that if $h_1, h_2\in\h$, then
$$
d_D(\Gamma(h),\Gamma(h_2))=d_H(h_1,h_2).
$$
\end{proof}
In \cite{meyer}, treating other kind of problems, a similar homeomorphism was established, between $\d$ and the set of elements in $\a$ with positive and invertible {\it real} part (i.e.   $-i\h$)

Recall that the group $\u(\theta_H)$ acts transitively on $\h$. This group is, in turn, isomorphic to $\u(\theta_D)$, which acts transitively on $\d$. Also recall (Proposition (\ref{equivPhi})), that these actions are equivariant for $\Phi_H$ and $\Phi_D$, respectively.
Then one has:
\begin{coro}
The actions of $\spa$ and $\spb$ on $\h$ and $\d$, respectively, are isometric. 
\end{coro}
\begin{proof}
The diffeomorphism $\Phi_H$, followed by the embedding of $\q_{\rho_H}$, is a composition of maps which are equivariant for the action of $\spa\subset Gl_2(\a)$. On the other hand, the action of this latter group on $Gl_2^+(\a)$ is isometric (\cite{cpr}, p.66).
\end{proof}

\begin{ejem}
Let us compute the $d_D$ distance between $0$ and $z$ in $\d$:
$$
d_D(0,z)=\| \log(\Phi_D(z)\rho_D)\|=\|\log \left( \begin{array}{cc} 2(1-z^*z)^{-1} & 2(1-z^*z)^{-1}z^* \\ 2z(1-z^*z)^{-1} & 2z(1-z^*z)^{-1}z^*+1 \end{array} \right) \|.
$$
Straightforward computations show that this matrix above can be factorized
$$
\left( \begin{array}{cc} 2(1-z^*z)^{-1} & 2(1-z^*z)^{-1}z^* \\ 2z(1-z^*z)^{-1} & 2z(1-z^*z)^{-1}z^*+1 \end{array} \right)=\Delta_1\Delta_2,
$$
where
$$
\Delta_1=\left( \begin{array}{cc} (1-z^*z)^{-1} & 0 \\ 0 & (1-zz^*)^{-1} \end{array} \right)\ \  \hbox{ and } \ \ \  \Delta_2=\left( \begin{array}{cc} 1+z^*z & 2z^* \\ 2z & 1+zz^* \end{array} \right).
$$
(A key fact in this computation is that $z(1-z^*z)^{-1}=(1-zz^*)^{-1}z$). 
Denote by $\Omega=\left(\begin{array}{cc} 0 & z^* \\ z & 0 \end{array}\right)$. Note that $\Omega^*=\Omega$ and $\|\Omega\|<1$. Then
$$
\Delta_1=(1-\Omega^2)^{-1} \ \hbox{ and } \Delta_2=(1+\Omega)^2.
$$
Therefore $\log(\Delta_1\Delta_2)=\log(1+\Omega)-\log(1-\Omega)$. Using the power series 
$$
\log(1+t)-\log(1-t)=2\sum_{k=0}^\infty \frac{t^{2k+1}}{2k+1}
$$
we get that  
$$
\log(1+\Omega)-\log(1-\Omega)=2\left(\begin{array}{cc} 0 & z^*\sum_{k=0}^\infty \frac{1}{2k+1} (zz^*)^k \\ z\sum_{k=0}^\infty \frac{1}{2k+1} (z^*z)^k & 0 \end{array} \right)
$$
The norm of this matrix equals
$$
\|\log(1+\Omega)-\log(1-\Omega)\|=\|(\log(1+\Omega)-\log(1-\Omega))^2\|^{1/2}.
$$
Note that $(\log(1+\Omega)-\log(1-\Omega))^2$ equals
$$
4\left( \begin{array}{cc}  (z^*\sum_{k=0}^\infty \frac{1}{2k+1} (zz^*)^k)(z\sum_{k=0}^\infty \frac{1}{2k+1} (z^*z)^k)   &  0  \\ 0 & (z\sum_{k=0}^\infty \frac{1}{2k+1} (z^*z)^k)( z^*\sum_{k=0}^\infty \frac{1}{2k+1} (zz^*)^k) \end{array} \right)
$$
$$
=4\left( \begin{array}{cc} z^*z(\sum_{k=0}^\infty \frac{1}{2k+1} (z^*z)^k)^2 & 0 \\ 0 & zz^*(\sum_{k=0}^\infty \frac{1}{2k+1} (zz^*)^k)^2\end{array} \right)
$$
$$
= \left( \begin{array}{cc} \log(1+|z|)-\log(1-|z|) & 0 \\ 0 & \log(1+|z^*|)-\log(1-|z^*|) \end{array}\right)^2.
$$
The square root of the norm of this matrix is
$$
\max\{ \|\log(1+|z|)-\log(1-|z|)\|, \|\log(1+|z^*|)-\log(1-|z^*|)\| \}.
$$
The function $f(t)=\log(1+t)-\log(1-t)$ is strictly increasing in $[0,1)$, with $f(0)=0$. Thus (using that $\||z|\|=\|z\|$),
$$
\|log(1+|z|)-\log(1-|z|)\|=\max\{|f(t)|: t\in\sigma(|z|)\}=f(\|z\|).
$$
Analogously, $\|log(1+|z^*|)-\log(1-|z^*|)\|=f(\|z\|)$.
Then
$$
d_D(0,z)=\log( \frac{1+\|z\|}{1-\|z\|}).
$$
In the scalar case $\a=\mathbb{C}$, this norm  equals
$$
d_D(0,z)=\log\left(\frac{1+|z|}{1-|z|}\right),
$$
which is the Poincar\'e distance in the open unit disk $\mathbb{D}$.
\end{ejem}
\section{The covariant derivative in $\h$}

In this section we compute explicitly the covariant derivative induced by the reductive structure.
Recall the decomposition (\ref{horizontalH}) of the Banach-Lie algebra of $\u(\theta_H)$),
$$
\x=\x_{as}\oplus\x_{s}=\{\tilde{\beta}\in M_2(\a)_{as}: \tilde{\beta}J=J\tilde{\beta}\}\oplus \{\tilde{\gamma}\in M_2(\a)_s: \tilde{\gamma}J=-J\tilde{\gamma}\}.
$$
Elements $X\in\x$ are of the form $X=X_0+X_h$,
$$
X=\left( \begin{array}{cc} x_{11} &  x_{12}  \\  -x_{12} & -x_{11} \end{array}\right) + \left( \begin{array}{cc} \alpha &  \beta  \\  \beta & -\alpha \end{array}\right)
$$
with $x^*_{11}=-x_{11}$, and all other entries selfadjoint. The left hand subspace $\x_{as}$ is the Banach-Lie algebra  of the isotropy group of the action at the element $i\in\h$. The right hand subspace $\x_s$ is the {\it horizontal} space at this point.

The computation of the covariant derivative  will be done in several steps. 

\underline{\bf Step 1.} First we compute the differential of the map $\pi_i:\spa\to \h$, $\pi_i(\tilde{g})=\tilde{g}\cdot i$, at the identity $1\in\spa$. Recall that $\pi_i(\tilde{g})=(g_{22} i + g_{21}(g_{12} i + g_{11})^{-1}$. Then, differentiating at $1$, we get
$$
d(\pi_i)_1(\tilde{\gamma})=\gamma_{21}+\gamma_{12}+i(\gamma_{22}-\gamma_{11}).
$$
If $\tilde{\gamma}$ is horizontal, $d(\pi_i)_1(\tilde{\gamma})= 2\gamma_{12}-2i \gamma_{11}$. Then:
\begin{teo}
The $1$-form of the reductive connection at $i\in\h$ is 
$$
\kappa_i(\zeta)=\frac12 \left( \begin{array}{cc} -\Upsilon & \chi \\ \chi & \Upsilon \end{array}\right) \ , \hbox{ if } \zeta=\chi+i \Upsilon.
$$
\end{teo}

\underline{\bf Step 2.}
For any given $h=x+i y$, one can find an element $\tilde{b}\in\b\subset\spa$ (the Borel subgroup of $\spa$) such that $\tilde{b}\cdot i =h$. For instance,
$$
\tilde{b}=\left( \begin{array}{cc} y^{-1/2} & 0 \\ xy^{-1/2} & y^{1/2} \end{array} \right) \ \hbox{ with inverse } \ \tilde{b}^{-1}=\left( \begin{array}{cc} y^{1/2} & 0 \\ -y^{-1/2}x & y^{-1/2} \end{array} \right).
$$ 
Straightforward computations show that $\tilde{b}\cdot i=h$  (and $\tilde{b}^{-1}\cdot h=i$).

\underline{\bf Step 3}  Let us compute now the differential of the action of $\tilde{g}\in\spa$ on tangent vectors of $\h$. If $\frac{d}{dt}h=\zeta$,
$$
\frac{d}{dt} (\tilde{g}\cdot h)= (g_{22}- w g_{12})\zeta(g_{12} h=g_{11})^{-1},
$$
where $w=\tilde{g}\cdot h=(g_{22}h+g_{21})(g_{12}z+g_{11})^{-1}$. 
Using the transformation in {\bf Step 2}, we can carry a tangent vector $\zeta$ at $h$ to the tangent vector $y^{-1/2}\zeta y^{-1/2}$ at $i$. If $\zeta+\chi+i \Upsilon$, 
$$
\kappa_i(y^{-1/2}\zeta y^{-1/2})+\frac12 y^{-1/2}\left( \begin{array}{cc} -\Upsilon & \chi \\ \chi & \Upsilon \end{array} \right) y^{-1/2}.
$$

\underline{ \bf Step 4} To obtain $\kappa_h(\zeta)$, we use the inner automorphism $Ad_{\tilde{b}}$:
$$
\left( \begin{array}{cc} y^{-1/2} & 0 \\ xy^{-1/2} & y^{-1/2} \end{array}\right) \{ \frac12 y^{-1/2} \left( \begin{array}{cc} -\Upsilon & \chi \\ \chi & \Upsilon  \end{array}\right) y^{-1/2}\} \left( \begin{array}{cc} y^{1/2} & 0 \\ -y^{-1/2}x & y^{-1/2} \end{array}\right)
$$
$$
= \frac12 \left( \begin{array}{cc} -y^{-1}\chi y^{-1}x  &  y^{-1}\chi y^{-1} \\ \chi-xy^{-1}\chi y^{-1}x & xy^{-1}\chi y^{-1} \end{array}\right)+\frac12 \left( \begin{array}{cc} -y^{-1}\Upsilon & 0 \\ -xy^{-1}\Upsilon-\Upsilon y^{-1}x & \Upsilon y^{-1} \end{array}\right).
$$

\underline{\bf Step 5} Next, we compute the covariant derivative of the reductive connection in $\h$, at the point $i\in\h$. Consider the curve $h(t)=x(t)+i y(t)$ in $\h$, with $z(0)=i$, and the vector field $\zeta=\chi+i\Upsilon$ defined on a neighbourhood of $i$.  
Denote by $\frac{D\zeta}{dt}|_{t=0}$ the covariant derivative at $i\in\h$. According to \cite{espacioshomo}, one has that
$$
\kappa_i(\frac{D\zeta}{dt}|_{t=0})=\frac{d}{dt}\kappa_{h(t)}Z(h(t))|_{t=0}+ [\kappa_i(Z(i)),\kappa_i(\frac{d}{dt}h(t)|_{t=0}].
$$
To lighten the notation, we shall denote by $\kappa(\zeta)', h'$ the usual derivatives at $t=0$. So the formula above reads
$$
\kappa_i(\frac{ D\zeta}{dt}|_{t=0})= \kappa(\zeta)'+[\kappa_i(Z(i)),\kappa_i(h')].
$$
Then
$$
\kappa(\zeta)'=\frac12\left(\begin{array}{cc} -\Upsilon' & \chi' \\ \chi' & \Upsilon' \end{array} \right) + \frac12   \left(\begin{array}{cc} y'\Upsilon-\chi x' & -y'\chi-\chi y' \\ -x'\chi-\Upsilon x' & x'\chi-\Upsilon y' \end{array} \right)
$$
Let us write the sum of right hand matrix above with the bracket $[\kappa_i(Z(i)),\kappa_i(h')]$. After strenuous but elementary computations one gets
$$
\frac14 \left(\begin{array}{cc} \Upsilon y'-\chi x' & -\Upsilon x'-\chi y' \\ -\chi y'-\Upsilon x' & \chi x'-\Upsilon y' \end{array}\right) + \frac14\left(\begin{array}{cc} y' \Upsilon-x'\chi & -y'\chi -x'\Upsilon \\ -x'\Upsilon -y'\chi & x'\chi-y'\Upsilon \end{array} \right) .
$$
Now we must apply $\kappa_i^{-1}$, or else realize that the above term is the value of $\kappa_i$ at
$$
-Re(x'\Upsilon+y'\chi)+i Re(x'\chi-y'\Upsilon).
$$
Therefore,
$$
\frac{d\zeta}{dt}|_{t=0}=\chi'+i\Upsilon'+\{-Re(x'\Upsilon+y'\chi)+i Re(x'\chi-y'\Upsilon)\}.
$$
\underline{\bf Step 6} Let $h_0=x_0+i y_0\in\h$, and $h(t)=x(t)+i y(t)$ be a smooth curve in $\h$ with $h(0)=h_0$. Let $\zeta=\chi+i\Upsilon$ be a smooth vector field defined on a neighbourhood of $\h_0$. Let us compute   $\frac{D \zeta}{dt}|_{t=0}$. To do this, we shall use the invariance of the connection under the action of the group $\spa$, and the fact that we know this formula in the case $h_0=i$. As seen in Step 2, the (explicit) element $\tilde{b}$ defined there, performs $\tilde{b}^{-1}\cdot h_0=i$. Thus we can carry the data to the point $i$, perform the covariant derivative, and translate it back to $h_0$ with the action (note for instance, that $y_0^{1/2}Re(u)y_0^{1/2}=Re(y_0^{1/2}uy_0^{1/2}$, and so forth). We get:

\begin{equation}\label{derivada covariante}
\frac{D\zeta}{dt}=\zeta'-Re(x'y_0^{-1}\Upsilon+y'y_0^{-1}\chi)+i Re(x'y_0^{-1}\chi-y'y_0^{-1}\Upsilon).
\end{equation}
We devote the next section to describe examples of geodesics in the halfspace model $\h$:

\section{Examples of geodesics in $\h$}

In terms of  the decomposition (\ref{horizontalH}) of the Banach-Lie algebra of $\u(\theta_H)$,
$$
\x=\x_{as}\oplus\x_{s}=\{\tilde{\beta}\in M_2(\a)_{as}: \tilde{\beta}J=J\tilde{\beta}\}\oplus \{\tilde{\gamma}\in M_2(\a)_s: \tilde{\gamma}J=-J\tilde{\gamma}\}.
$$
geodesics of $\h$ starting at $i$  for $t=0$ have the form
$$
\delta(t)=e^{tX_h}\cdot i,
$$
where $X_h$ is a horizontal element, i.e., an antihermitian element in $M_2(\a)$ of the  form
$$
X_h=\left( \begin{array}{cc} \alpha &  \beta  \\  \beta & -\alpha \end{array}\right),
$$
with $\alpha^*=-\alpha$, $\beta^*=-\beta$.
We shall compute these geodesics in two particular cases: 
\begin{ejem}
Suppose that the entries $\alpha$ and $\beta$ in $X_h$  commute. Put $\gamma=(\alpha^2+\beta^2)^{1/2}$. The element $\gamma$ may not be invertible; however, if we make the assumption that $\a$ is weakly closed (i.e., a von Neumann algebra), then there exist unique selfadjoint elements $x$ and $y$ in $\a$ such that 
$x\gamma=\gamma x=\alpha$ and $y\gamma=\gamma y=\beta$. Moreover, there exists a selfadjoint element $\chi\in\a$ such that  $x=\cos(\chi)$ and $y=\sin(\chi)$. Accordingly, $X_h/\gamma=\left( \begin{array}{cc} \cos(\chi) & \sin(\chi) \\ \sin(\chi) & -\cos(\chi) \end{array} \right)$. Put ${\bf 1}_2=\left(\begin{array}{cc} 1 & 0 \\ 0 & 1 \end{array} \right)$.  Straightforward computations show that for $n\ge 0$
$$
(tX_h)^{2n}=(t\gamma)^{2n}{\bf 1}_2 \ \hbox{ and } \ (tX_h)^{2n+1}=(t\gamma)^{2n+1} X_h/\gamma .
$$
Then 
$$
e^{tX_h}=\cosh (t\gamma) {\bf 1}_2+  \sinh(t\gamma) X_h/\gamma=\left(\begin{array}{cc} \cosh(t\gamma)+\cos(\chi) \sinh(t\gamma) & \sin(\chi) \sinh(t\gamma) \\ \sin(\chi) \sinh(t\gamma) & \cosh(t\gamma)-\cos(\chi) \sinh(t\gamma) \end{array} \right).
$$
Therefore $\delta(t)=e^{tX_h}\cdot i$ equals
$$
\left(\sin(\chi) \sinh(t\gamma)+i (\cosh(t\gamma)-\cos(\chi)\right)\left(\cosh(t\gamma) +\cos(\chi)\sinh(t\gamma) + i \sin(\chi) \sinh(t\gamma)\right)^{-1}.
$$
After straightforward calculations, involving well known identities concerning $\cosh$ and $\sinh$, we arrive at
\begin{equation}\label{geodesica si conmutan}
\delta(t)=\left(\sin(\chi)\sinh(2t\gamma)+i\right)\left(\cosh(2t\gamma)+\cos(\chi)\sinh(2t\gamma)\right)^{-1}.
\end{equation}
Since all elements involved commute and $\xi,\gamma$ are selfadjoint, it follows that 
$$
Re(\delta(t))=\sin(\chi)\sinh(2t\gamma)\left(\cosh(2t\gamma)+\cos(\chi)\sinh(2t\gamma)\right)^{-1}
$$ 
and 
$$
Im(\delta(t))=\left(\cosh(2t\gamma)+\cos(\chi)\sinh(2t\gamma)\right)^{-1}.
$$
Therefore, if one regards $\delta$ as a geodesic in $TG^+$, i.e. $\delta=(Re(\delta),Im(\delta))$,  it is given by
$$
\delta(t)=\left( \sin(\chi)\sinh(2t\gamma)\left(\cosh(2t\gamma)+\cos(\chi)\sinh(2t\gamma)\right)^{-1}, \left(\cosh(2t\gamma)+\cos(\chi)\sinh(2t\gamma)\right)^{-1}\right).
$$

Additionally, if there exists $\mu\in\a$, $\mu=\mu^*$,  such that $\mu\beta=-\alpha$ (for instance, if $\beta$ is invertible), then
$$
(Re(\delta(t)-\mu)^2+(Im(\delta(t))^2=\mu^2+1, 
$$
that is, geodesics $\delta$  of the Poincar\'e halfspace (with commuting $\alpha,\beta$) satisfy the equation of an $\a$-valued circle, centered in the real axis (at $\mu^*=\mu$) with radius $(1+\mu^2)^{1/2}$.
\end{ejem}
Another special case which can be explicitly computed occurs when $\alpha$ and $\beta$ anti-commute: $\alpha\beta=-\beta\alpha$
\begin{ejem}
Suppose now that the entries $\alpha$, $\beta$ in $X_h$ anti-commute. Then $\left( \begin{array}{cc} \alpha & 0 \\ 0 & -\alpha \end{array}\right)$ and  $\left( \begin{array}{cc} 0 & \beta \\ \beta & 0 \end{array}\right)$  commute.
Thus
$$
e^{tX_h}=e^{t\left( \begin{array}{cc} \alpha & 0 \\ 0 & -\alpha \end{array}\right)}e^{t\left( \begin{array}{cc} 0 & \beta \\ \beta & 0 \end{array}\right)}=\left( \begin{array}{cc} e^{t\alpha} & 0 \\ 0 & e^{-t\alpha} \end{array}\right)\left( \begin{array}{cc} \cosh(t\beta) & \sinh(t\beta) \\ \sinh(t\beta) & \cosh(t\beta) \end{array}\right).
$$
Therefore
$$
\delta(t)=e^{tX_h}\cdot i=e^{-t\alpha}( \sinh(t\beta)+i\cosh(t\beta))(\cosh(t\beta) + i\sinh(t\beta))^{-1}e^{-t\alpha}.
$$
By straightforward computations, 
$$
(\sinh(t\beta)+i\cosh(t\beta))(\cosh(t\beta) + i\sinh(t\beta))^{-1}=(\sinh(2t\beta)+i)\cosh^{-1}(2t\beta).
$$
Since $\cosh$ is even, $\cosh(2t\beta)$ commutes with $e^{t\alpha}$, analogously, since $\sinh$ is odd, $e^{-t\alpha}\sinh(2t\beta)=\sinh(2t\beta)e^{t\alpha}$.
Then
$$
\delta(t)=\sinh(2t\beta)\cosh^{-1}(2t\beta)+i\cosh^{-1}(2t\beta)e^{-2t\alpha}.
$$

Clearly, this is the real/imaginary part decomposition of $\delta$. Note that the real part of $\delta$ does not depend on $\alpha$. Thus, as a curve in $TG^+$, this geodesic is given by
$$
\delta(t)=\left( \sinh(2t\beta)\cosh^{-1}(2t\beta), \cosh^{-1}(2t\beta)e^{-2t\alpha}\right).
$$

Suppose that $\alpha, \beta$ have a polar decompositions in $\a$, $\alpha=\mu|\alpha|=|\alpha|\mu$, $\beta=\nu|\beta|=|\beta|\nu$.  Then after elementary calculations, one has that the imaginary and real parts of $\delta$ satisfy the equation
$$
(Re(\delta-i\mu)^2+(\nu Im(\delta))^2=0.
$$

\end{ejem}

\section{Appendix: The Poincar\'e half-space of a Hilbertizable space.}

A complex locally convex topological vector space $\hh$ is {\it Hilbertizable} if there exists an inner product $\beta$ which makes $\hh$ a Hilbert space. In this section we fix such a space $\hh$.  We shall denote by $\s(\hh)$ the space of all sesquilinear forms $\sigma$ on $\hh$, which are continuous in both variables. We consider in $\s(\hh)$ the topology whose basis of neighbourhoods of the origin are the sets
$$
W(V,\epsilon)=\{ \sigma\in\s(\hh): |\sigma(\xi,\eta)|<\epsilon, \xi,\eta\in V\},
$$
where $V$ is a neighbourhoog of $0$ in $\hh$ and $\epsilon>0$. Let us denote by $\ii(\hh)$ the set of all $\beta\in\s(\hh)$ which are positive definite and which reproduce the topology of $\hh$.

There is a natural involution in $\s(\hh)$, which we shall call the {\it conjugation} in $\s(\hh)$, which is givan by
$$
\sigma^c(\xi,\eta)=\overline{\sigma(\eta,\xi)}.
$$
Then $\s(\hh)$ decomposes as
$$
\s(\hh)=\s_0(\hh)\oplus \s_1(\hh),
$$
the Hermitian ($\sigma_0^c=\sigma_0$) and anti-Hermitian ($\sigma^c_1=-\sigma_1$) forms, respectively. The set $\ii(\hh)$ is an open subset of $\s_0(\hh)$.

Denote by $\ele(\hh)$ the algebra of continuous linear operators acting in $\hh$. $\ele(\hh)$ is a topological algebra, with the topology given by 
$$
W(V,V')=\{ a\in\ele(\hh): a(\xi)\in V' \hbox{ for } \xi\in V\}
$$
as a system  of neighbourhoods of $0\in\ele(\hh)$, for $V,V'$ neighbourhoods of $0\in\hh$.

Given $\beta\in\ii(\hh)$, $\hh$ becomes a Hilbert space, we shall denote it by $\hh_\beta$. Likewise, $\ele(\hh)$ becomes a C$^*$-algebra, which will be denoted by $\ele_\beta(\hh)$.

\subsection{Charts in $\s(\hh)$}
Given $\beta\in\ii(\hh)$, put
$$
\Phi_\beta:\ele_\beta\to \s(\hh) \ , \ \ \Phi_\beta(a)(\xi,\eta)=\beta(a\xi,\eta).
$$
By Riesz' Theorem, it is clear that $\Phi_\beta$ is a bijection. We shall call $\Phi_\beta$ the {\it chart for} $\s(\hh)$ {\it centered at} $\beta$. Denote by $G(\hh)$ the group of bijective elements of $\ele(\hh)$. The group $G(\hh)$ acts on $\s(\hh)$ as  changes of variables, that is, if $g\in G(\hh)$ and $ \sigma\in\s(\hh)$,
$$
L_g\sigma(\xi,\eta)=\sigma(g^{-1}\xi,g^{-1}\eta).
$$
It is apparent that this action restricts to an action of $G(\hh)$ on $\ii(\hh)$, and that  it is transitive on $\ii(\hh)$.

Let us describe a change of charts by means on an element $g\in G(\hh)$. Let $\beta$ and $\tilde{\beta}=L_g\beta$. Then the following diagram commutes
\begin{equation}\label{cuadrado}
\begin{array}{lll}
\ele_\beta(\hh) &\stackrel{{\mathbb L}_g}{\longrightarrow} & \ele_\beta(\hh) \\
\downarrow Ad_g &   & \downarrow \Phi_\beta \\
\ele_{\tilde{\beta}}(\hh)  & \stackrel{\Phi_{\tilde{\beta}}}{\longrightarrow} & \s(\hh)
\end{array} .
\end{equation}
Here 
\begin{itemize}
\item
${\mathbb L}_g:\ele_\beta(\hh)\to \ele_\beta(\hh)$ is the action given by ${\mathbb L}_ga=\hat{g}ag^{-1}$, where $\hat{g}=(g^{-1})^*$, with $*$ the involution of $\ele_\beta(\hh)$. This map $\mathbb{L}_g$ is a $*$-preserving linear isomorphism. It is  also an isomorphism of the set $\ele_\beta^+(\hh)$  of positive invertible operators, preserving its metric and its linear connection (see Section 6).
\item
$Ad_g:\ele_\beta(\hh)\to \ele_{\tilde{\beta}}(\hh)$ is a C$^*$-algebra isomorphism.
\end{itemize}

The diagram (\ref{cuadrado}) can be read as follows:
$$
\Phi_{L_g\beta}=\Phi_\beta \mathbb{L}_g Ad_{g^{-1}}.
$$
As was noted in Section 6, the group $G(\hh)$ acts on $\ele_\beta^+(\hh)$ by means of the action $\mathbb{L}_g$. $G(\hh)$ acts also on $\ii(\hh)$ as noted above.  The commutativity of the diagram (\ref{cuadrado}) implies that  $\Phi_\beta$ intertwines both actions:
$$
\Phi_\beta:\ele_\beta^+(\hh) \to \ii(\hh) \ ,  \  L_g\Phi_\beta=\Phi_\beta \mathbb{L}_g.
$$
\subsection{The Poincar\'e half-space}
We shall denote the Poincar\'e half-space of $\hh$  by ${\bf \Pi}(\hh)$, which is the set
$$
{\bf \Pi}(\hh)=\{\sigma\in\s(\hh) : Im(\sigma)\in\ii(\hh)\},
$$
where $Re(\sigma)=\frac12(\sigma+\sigma^c)$ and $Im(\sigma)=\frac{i}{2}(\sigma^c-\sigma)$. We define charts in ${\bf \Pi}(\hh)$. Let $\beta\in\ii(\hh)$. Clearly $\Phi_\beta(a^*)=\Phi_\beta(a)^c$. Therefore $\Phi_\beta$ maps the Poincar\'e halfspace space $\h(\ele_\beta(\hh))$ of the C$^*$-algebra $\ele_\beta(\hh)$ as defined in Section 1, onto   ${\bf \Pi}(\hh)$.

The purpose of this Appendix is to introduce the geometry  of the {\it bundle of observables} associated with the space of metrics $\ii(\hh)$ of the Hilbertizable space $\hh$. Specifically, the bundle
$$
\oo\to \ii(\hh),
$$  
where the fiber $\oo_\beta$  over $\beta\in\ii(\hh)$ is the vector space
$$
\oo_\beta=\{ a\in\ele_\beta(\hh): a^*=a\}.
$$
We claim that the natural way to present the bundle of observables is as the tangent bundle $T\ii(\hh)$.

But the space $T\ii(\hh)$ can be identified with the Poincar\'e half-space ${\bf \Pi}(\hh)$ as follows. A tangent vector $X\in (T\ii(\hh)_\beta$ is, canonically, an element of $\s_0(\hh)$. Therefore the map
$$
(\beta,X)\in(T\ii(\hh)_\beta \longleftrightarrow X+i\beta \in {\bf \Pi}(\hh)
$$
identifies $T\ii(\hh)$  with ${\bf \Pi}(\hh)$ in a natural way.

\subsection{The group of movements in ${\bf \Pi}(\hh)$.}

Let $\Phi_\beta$ and  $\Phi_{\tilde{\beta}}$, with $\tilde{\beta}=L_g\beta$ for $g\in G(\hh)$, be  a pair of charts. We shall denote by $\phi$ the change of charts given by
$$
\begin{array}{lll} \ele_\beta(\hh) & \stackrel{\Phi_\beta}{\to} & \s(\hh)  \\  \uparrow  \psi &  \stackrel{\Phi_{\tilde{\beta}} }{\nearrow} & 
\\ \ele_{\tilde{\beta}}(\hh) &  &   \end{array}\  , \ \  \ \Phi_\beta^{-1}\Phi_{\tilde{\beta}}=\mathbb{L}_g Ad_{g^{-1}}=\psi.
$$ 
The C$^*$-isomorphism $Ad_{g^{-1}}:\ele_{\tilde{\beta}}(\hh)\to \ele_\beta(\hh)$ enables one to relate any construction done in both algebras. For instance, the forms $\theta_H$ and their unitary groups $\spa$. Consider the group $\spa_\beta$, which  acts on the half-space ${\bf \Pi}(\ele_\beta(\hh))\subset \ele_\beta(\hh)$. If $h\in\spa_\beta$ and $z\in{\bf \Pi}(\ele_\beta(\hh))$, denote by  $\Lambda_hz$ the action of $h$ on $z$. We have the following diagram:
$$
\begin{array}{ccc}
{\bf \Pi}(\ele_\beta(\hh)) & \stackrel{\psi \Lambda_{\tilde{h}} \psi^{-1}}{\longrightarrow} & {\bf \Pi}(\ele_\beta(\hh)) \\
\uparrow \psi & & \uparrow \psi \\
{\bf \Pi}(\ele_{\tilde{\beta}}(\hh)) & \stackrel{\Lambda_{\tilde{h}}}{\longrightarrow} &  {\bf \Pi}(\ele_{\tilde{\beta}}(\hh))
\end{array} .
$$
Here $\Lambda_{\tilde{h}}$ denotes  the action of $\tilde{h}\in\spa_{\tilde{\beta}}$ on ${\bf \Pi}(\ele_{\tilde{\beta}}(\hh))$.
We want to exhibit how the transition maps $\psi$ translate this action on ${\bf \Pi}(\ele_{\beta}(\hh))$. Note that
\begin{equation}\label{lambdas}
\psi \Lambda_{\tilde{h}}\psi^{-1}=\mathbb{L}_g Ad_{g^{-1}} \Lambda_{\tilde{h}} Ad_g \mathbb{L}_{g^{-1}}=\mathbb{L}_g \Lambda_h \mathbb{L}_{g^{-1}}.
\end{equation}

Denote by $G_\beta(\hh)$ the group of invertible operators in $\ele_\beta(\hh)$. There is a representation 
$$
u: G_\beta(\hh)\to \spa_\beta \ , \ \  u(g)=\left( \begin{array}{cc} g & 0 \\ 0 & \hat{g} \end{array} \right).
$$
(Recall that $\hat{g}=(g^{-1})^*$). Using this representation  we may write (\ref{lambdas}) above as
$$
\psi \Lambda_{\tilde{h}}\psi^{-1}=u(g)\Lambda_h u(g)^{-1}.
$$
Moreover, if we denote by $\mathbb{A}d_g:\spa_\beta\to \spa_{\tilde{\beta}}$,
$$
\mathbb{A}d_g=u(g)hu(g)^{-1},
$$
 it is straightforward to verify that
\begin{equation}\label{ads}
\psi \Lambda_{\tilde{h}}\psi^{-1}=\mathbb{A}d_g\lambda_{\mathbb{A}d_{g^{-1}}}\tilde{h}.
\end{equation}
Therefore we have obtained the following construction:
\begin{itemize}
\item
To each $\beta\in\ii(\hh)$ we associate the group $\spa_\beta$.
\item
If $\tilde{\beta}=L_g\beta$, to the pair $(\beta,\tilde{\beta})$ we associate the group  isomorphism 
$$
\omega_{\beta \tilde{\beta}}: \spa_{\tilde{\beta}} \to \spa_\beta , \omega_{\tilde{\beta},\beta}(\tilde{h})=\mathbb{A}d_g\lambda_{\mathbb{A}d_{g^{-1}}}\tilde{h}.
$$ 
Remarkably, this isomorphism $\omega_{\beta \tilde{\beta}}$ does not depend on the choice of $g$: if $\tilde{\beta}=L_g\beta=L_{g'}\beta$, then $g$ and $g'$ give rise to the same isomorphism. This is another straightforward computation left to the reader.
\end{itemize}
We shall refer to this system of groups and group isomorphism as the {\it group of movements} ${\cal U}(\theta)$  of the Poincar\'e half-space ${\bf \Pi}(\hh)$. This group of movements acts on ${\bf \Pi}(\hh)$ and is independent on the choice of coordinates. Thus,  ${\bf \Pi}(\hh)$ is a homogeneous space of the group of movements ${\cal U}(\theta)$.

\subsection{The relative case}

So far, the building blocks of this construction are a Hilbertizable space $\hh$, the space of forms $\s(\hh)$ and its subset of inner products $\ii(\hh)$. Now we want to restrict these constructions to a subalgebra $\a\subset \ele(\hh)$. Let us define:

\begin{defi}
A {\it C$^*$-pair} is a pair $(\hh,\a)$,  where  $\hh$ is a Hilbertizable space and $\a\subset\ele(\hh)$ is a subalgebra, such that there exists $\beta\in\ii(\hh)$ such that $\a$ is a sub-C$^*$-algebra of $\ele_\beta(\hh)$. 

In such case, we say that $\beta$ is adapted to $\a$. Let us denote by
$$
\ii(\a)=\{\beta\in\ii(\hh): \beta \hbox{ is adapted to } \a\}.
$$
\end{defi}
Denote by $G_\a$ the invertible group of $\a$. Clearly $G_\a$ acts on $\ii(\hh)$, and $L_g\beta\in\ii(\a)$ if $\beta\in\ii(\a)$ and $g\in G_\a$. Denote by $\ii_{\cal O}(\a)\subset \ii(\a)$ an orbit of this action.  A {\it  C$^*$-triple} is a triple 
$$
(\hh,\a,\ii_{\cal O}(\a))
$$
consisting of a C$^*$-pair and a orbit.  With these data, we can repeat the former constructions, substituting  $\ele(\hh)$ by $\a$, and $\ii(\hh)$ by $\ii_{\cal O}(\a)$.

The Poincar\'e half-space ${\bf \Pi}(\a)$ of $\a$ is formed by the elements of $\s_{\cal O}(\a)$ with positive imaginary part.  The (restricted) group of movements ${\cal U}(\theta)$ acts on ${\bf \Pi}(\a)$ accordingly.

Esteban Andruchow \\
Instituto de Ciencias,  Universidad Nacional de Gral. Sarmiento,
\\
J.M. Gutierrez 1150,  (1613) Los Polvorines, Argentina
\\ 
and Instituto Argentino de Matem\'atica, `Alberto P. Calder\'on', CONICET, 
\\
Saavedra 15 3er. piso,
(1083) Buenos Aires, Argentina.
\\
e-mail: eandruch@ungs.edu.ar

\bigskip

Gustavo Corach\\
Instituto Argentino de Matem\'atica, `Alberto P. Calder\'on', CONICET,
\\
Saavedra 15 3er. piso, (1083) Buenos Aires, Argentina,
\\
and Depto. de Matem\'atica, Facultad de Ingenier\'\i a, Universidad de Buenos Aires, Argentina.
\\
e-mail: gcorach@fi.uba.ar

\bigskip

L\'azaro Recht \\
Departamento de Matem\'atica P y A, 
Universidad Sim\'on Bol\'\i var \\
Apartado 89000, Caracas 1080A, Venezuela  \\
e-mail: recht@usb.ve


\begin{thebibliography}{99}

\bibitem{atkin1} Atkin, C. J., The Finsler geometry of groups of isometries of Hilbert space, J. Austral. Math. Soc. Ser. A 42 (1987), 19 -222.

\bibitem{atkin2} Atkin, C. J.,The Finsler geometry of certain covering groups of operator groups, Hokkaido Math. J. 18 (1989),  45 -77.

\bibitem{beltita} Beltita, D., Smooth homogeneous structures in operator theory. Chapman \& Hall/CRC Monographs and Surveys in Pure and Applied Mathematics, 137. Chapman \& Hall/CRC, Boca Raton, FL, 2006.

\bibitem{bridson}  Bridson, M. R.; Haefliger, A., Metric spaces of non-positive curvature. Grundlehren der Mathematischen Wissenschaften [Fundamental Principles of Mathematical Sciences], 319. Springer-Verlag, Berlin, 1999.

\bibitem{burago} Burago, D.; Burago, Y.; Ivanov, S., A course in metric geometry. Graduate Studies in Mathematics, 33. American Mathematical Society, Providence, RI, 2001. 

\bibitem{cprIEOT}   Corach, G.;   Porta, H.; Recht, L., The geometry of the space of selfadjoint invertible elements in a $C^*$-algebra. Integral Equations Operator Theory 16 (1993), 333--359.

\bibitem{cpr}  Corach, G.; Porta, H.;  Recht, L., The geometry of spaces of projections in $C^*$-algebras, \textit{Adv. Math.} \textbf{101} (1993), 59--77.

\bibitem{cprIJM}  Corach, G.;  Porta, H.; Recht, L., Geodesics and operator means in the space of positive operators, Internat. J. Math. 4 (1993), 193--202.

\bibitem{cprILLINOIS}  Corach, G.;  Porta, H.; Recht, L., Convexity of the geodesic distance on spaces of positive operators. Illinois J. Math. 38 (1994), no. 1, 87--94.

\bibitem{harpe1} de la Harpe, P., Classical Banach-Lie algebras and Banach-Lie groups of operators in Hilbert space. Lecture Notes in Mathematics, Vol. 285. Springer-Verlag, Berlin-New York, 1972.

\bibitem{harpe2} de la Harpe, P., Classical groups and classical Lie algebras of operators.Operator algebras and applications, Part I (Kingston, Ont., 1980), pp. 477--513, Proc. Sympos. Pure Math., 38, Amer. Math. Soc., Providence, R.I., 1982.

\bibitem{gromov} Gromov, M., Metric structures for Riemannian and non-Riemannian spaces. Based on the 1981 French original. With appendices by M. Katz, P. Pansu and S. Semmes. Translated from the French by Sean Michael Bates. Reprint of the 2001 English edition. Modern Birkhäuser Classics. Birkhäuser Boston, Inc., Boston, MA, 2007.
 
\bibitem{pucho} Larotonda, A. R., Notas sobre variedades diferenciables. (Spanish) [[Notes on differentiable manifolds]] Notas de Geometr\'\i a y Topolog\'\i a [Notes on Geometry and Topology], 1. Universidad Nacional del Sur, Instituto de Matem\'atica, Bah\'\i a Blanca, 1980. 

\bibitem{espacioshomo} Mata-Lorenzo, L. E.; Recht, L., Infinite-dimensional homogeneous reductive spaces. Acta Cient. Venezolana 43 (1992), 76--90.

\bibitem{meyer} Meyer, R., Adjoining a unit to an operator algebra, J. Operator Theory 46 (2001), 281--288.

\bibitem{mostow}  Mostow, G. D., Some new decomposition theorems for semi-simple groups, Mem. Amer. Math. Soc. No. 14 (1955), 31--54.

\bibitem{prILLINOIS} Porta, H.; Recht, L., Geometric embeddings of operator spaces, Illinois J. Math. 40 (1996), 151--161.

\bibitem{raeburn} Raeburn, I., The relationship between a commutative Banach algebra and its maximal ideal space, J. Functional Analysis 25 (1977), 366-390.

\bibitem{siegel1}  Siegel, C.L., Topics in complex function theory. Vol. I: Elliptic functions and uniformization theory. Translated from the original German by A. Shenitzer and D. Solitar. Interscience Tracts in Pure and Applied Mathematics, No. 25 Wiley-Interscience A Division of John Wiley \& Sons, New York-London-Sydney 1969 {\rm ix}+186 pp.

\bibitem{siegel2} Siegel, C.L., Topics in complex function theory. Vol. II. Automorphic functions and abelian integrals. Translated from the German by A. Shenitzer and M. Tretkoff. With a preface by Wilhelm Magnus. Reprint of the 1971 edition. Wiley Classics Library. A Wiley-Interscience Publication. John Wiley \& Sons, Inc., New York, 1988. {\rm xii}+193 pp.

\bibitem{siegel3} Siegel, C.L., Topics in complex function theory. Vol. III. Abelian functions and modular functions of several variables. Translated from the German by E. Gottschling and M. Tretkoff. With a preface by Wilhelm Magnus. Reprint of the 1973 original. Wiley Classics Library. A Wiley-Interscience Publication. John Wiley \& Sons, Inc., New York, 1989. {\rm x}+244 pp.

\bibitem{upmeier} Upmeier, H., Symmetric Banach manifolds and Jordan $C\sp \ast$-algebras. North-Holland Mathematics Studies, 104. Notas de Matem\'atica [Mathematical Notes], 96. North-Holland Publishing Co., Amsterdam, 1985. 


\end{thebibliography}
\end{document}